\documentclass{amsart}

\usepackage{amsmath, amssymb, graphicx}
\usepackage{pst-all}
\usepackage[T1]{fontenc} 

\allowdisplaybreaks

\psset{unit=1mm}
\psset{fillcolor=white}
\psset{dotsep=1.5pt}
\psset{dash=1.5pt 1.5pt}
\psset{linewidth=0.8pt}
\psset{arrowsize=3.5pt}
\psset{doublesep=1pt}
\psset{coilwidth=1.2mm}
\psset{coilheight=2}
\psset{coilaspect=20}
\psset{coilarm=2}
\newpsstyle{double}{linewidth=0.5pt, doubleline=true}
\newpsstyle{etc}{linestyle=dotted}
\newpsstyle{exist}{linestyle=dashed}
\newpsstyle{thin}{linewidth=0.5pt}
\newpsstyle{thinexist}{linewidth=0.5pt, linestyle=dashed}

\numberwithin{equation}{subsection}

\theoremstyle{plain} 
\newtheorem{prop}[equation]{Proposition}
\newtheorem{coro}[equation]{Corollary}
\newtheorem{lemm}[equation]{Lemma}
\newtheorem*{theoA}{Theorem~A}
\newtheorem*{theoB}{Theorem~B}
\newtheorem*{theoC}{Theorem~C}

\theoremstyle{definition}
\newtheorem{defi}[equation]{Definition}
\newtheorem{exam}[equation]{Example}
\newtheorem{nota}[equation]{Notation}
\newtheorem*{rema*}{Remark}

\newcounter{ITEM}
\newcommand\ITEM[1]{\setcounter{ITEM}{#1}\leavevmode\hbox{\rm(\roman{ITEM})}}

\renewcommand\aa{a}
\newcommand\ai[1]{121...[#1]}
\newcommand\aii[1]{212...[#1]}
\newcommand\Att{\widetilde{\mathrm{A}}_2}

\newcommand\bb{b}
\newcommand\BP[1]{B^{\scriptscriptstyle+}_{#1}}% Braids

\newcommand\card{\mathtt{\#}}
\newcommand\cc{c}
\newcommand\CC{C}
\newcommand\comp{\mathbin{\vcenter{\hbox{$\scriptscriptstyle\circ$}}}}%composition

\newcommand\dd{d}
\newcommand\DD{D}
\newcommand\dive{\preccurlyeq\nobreak}

\newcommand\ee{e}% neutral letter
\newcommand\ev{\textsc{ev}}

\newcommand\ff{f}
\newcommand\FF{F}
\newcommand\fl{\rightarrow}
\newcommand\ft{\phi}

\renewcommand\ge{\geqslant}
\renewcommand\gg{g}

\newcommand\hh{h}
\newcommand\HH{H}
\newcommand\HS[1]{\hspace{#1ex}}

\newcommand\ii{i}

\newcommand\jj{j}

\newcommand\kk{k}

\renewcommand\le{\leqslant}
\newcommand\LG[1]{\Vert#1\Vert}
\newcommand\lLex{<^*}

\newcommand\mm{m}
\newcommand\MM{M}

\newcommand\NF{\textsc{nf}}
\newcommand\nm{N}
\newcommand\nmr{\overline{N}}% restriction a S^2
\newcommand\nmR[1]{\nm^{[#1]}}
\newcommand\nmrS{\nmr}% Garside case
\newcommand\nmS{\nm}% Garside case
\newcommand\nn{n}
\newcommand\NNNN{\mathbb{N}}

\newcommand\one{\underline{1}}

\newcommand\pdots{\hspace{0.2ex}{\cdot}{\cdot}{\cdot}\hspace{0.2ex}}
\newcommand\Pow[2]{#1^{[#2]}}%Power
\newcommand\pp{p}
\newcommand\PP{P}
\newcommand\PRESp[2]{\langle#1\,\vert\, #2\rangle^{\scriptscriptstyle\!+}\!}

\newcommand\qq{q}

\newcommand\rr{r}
\newcommand\RR{R}

\newcommand\sep{\HS{0.05}{\vert}\HS{0.05}}% path separator
\newcommand\sh{\mathrm{sh}}%shift
\newcommand\sig[1]{\sigma_{\hspace{-0.2ex}#1}^{\null}}
\newcommand\sigg[2]{\sigma_{\hspace{-0.2ex}#1}^{#2}}
\renewcommand\sp{\HS{0.15}\raise1pt\hbox{$\scriptstyle\vert$}\HS{0.15}}
\renewcommand\ss{s}
\renewcommand\SS{S}

\newcommand\ssm[1]{\ss_{#1,+}}
\newcommand\SSm{\SS_\ee}% S with $\ee$ removed
\newcommand\ssp[1]{\ss_{#1,-}}
\newcommand\SSp{{\SS^\ee}}

\newcommand\Sym{\mathfrak{S}}

\renewcommand\tt{t}
\newcommand\TT{T}
\newcommand\tta{\mathtt{a}}
\newcommand\ttb{\mathtt{b}}
\newcommand\ttc{\mathtt{c}}
\newcommand\ttd{\mathtt{d}}
\newcommand\ttx{\mathtt{x}}
\newcommand\tty{\mathtt{y}}
\newcommand\ttz{\mathtt{z}}

\newcommand\uu{u}

\renewcommand\vdots{\vert\pdots\vert}
\def\VR(#1,#2){\vrule width0pt height#1mm depth#2mm}
\newcommand\vv{v}
\newcommand\VVV[1]{\mathrm{VV}_{\!\!#1}}

\newcommand\wdots{, ...\HS{0.2},}
\newcommand\ww{w}
\newcommand\WW{W}
\newcommand\wwh{\widehat\ww}
\newcommand\www{\overrightarrow\ww}

\newcommand\XX{X}

\title{Quadratic normalisation in monoids} 

\author{Patrick DEHORNOY}
\address{Laboratoire de Math\'ematiques Nicolas Oresme,
CNRS UMR 6139, Universit\'e de Caen, 14032 Caen, France}
\email{patrick.dehornoy@unicaen.fr}
\urladdr{www.math.unicaen.fr/\!\hbox{$\sim$}dehornoy}

\author{Yves GUIRAUD}
\address{INRIA $\pi r^2$, Laboratoire Preuves, Programmes et Syst\`emes, CNRS UMR 7126, Universit\'e Paris~7, Case 7014, 75205 Paris Cedex 13, France}
\email{yves.guiraud@pps.univ-paris-diderot.fr}
\urladdr{www.pps.univ-paris-diderot.fr/~guiraud}



\begin{document}

\begin{abstract}
In the general context of presentations of monoids, we study normalisation processes that are determined by their restriction to length-two words. Garside's greedy normal forms and quadratic convergent rewriting systems, in particular those associated with the plactic monoids, are typical examples. Having introduced a parameter, called the class and measuring the complexity of the normalisation of length-three words, we analyse the normalisation of longer words and describe a number of possible behaviours. We fully axiomatise normalisations of class~$(4, 3)$, show the convergence of the associated rewriting systems, and characterise those deriving from a Garside family.
\end{abstract}

\maketitle

\section{Introduction}

A normal form for a monoid~$\MM$, with a specified generating subfamily~$\SS$, is a map that assigns to each element of~$\MM$ a distinguished representative word over~$\SS$. Our aim in this paper is to investigate a certain type of such normal forms and, more precisely, the associated normalisation processes, that is, the syntactic transformations that lead from an arbitrary word to a normal word. Here we restrict to geodesic normal forms, which select representatives of minimal length, and investigate the quadratic case, that is, when some locality conditions are satisfied: that a word is normal if, and only if, each of its length-two factors are normal, and that one can always transform a word into a normal word by a finite sequence of steps, each of which consists in normalising a length-two factor. 

This general framework includes two well-known classes of normalisation processes: those associated with Garside families as investigated in~\cite{Dig} and~\cite{Garside}, building on the seminal example of the greedy normal form in Artin's braid monoids~\cite{Adj, ElM, Eps}, and those associated with quadratic rewriting systems as investigated for instance in~\cite{GaussentGuiraudMalbos} for Artin monoids and in~\cite{BokutChenChenLi, CainGrayMalheiro} for plactic monoids. So our current development can be seen as an effort to unify various approaches and understand their common features. This program is made natural by the observation that, in spite of their unrelated definitions, the normalisation processes arising in the above mentioned situations share common mechanisms: for instance, in each case, a length-three word can be normalised in three steps, successively normalising the length-two factors in position~$2$-$3$, then in position~$1$-$2$, and in position~$2$-$3$ again.

D.\,Krammer's ideas had a seminal influence in our approach, in particular for the connection between normalisation and the monoid underlying Subsection~\ref{SS:Axiomatisation}, which he investigated in~\cite{KraArt}. A similar connection was independently discovered by A.\,Hess and V.\,Ozornova in~\cite{Hes, Ozo, HeOz}, partly building on unpublished work by M.\,Rodenhausen. Our current approach is close to theirs in the case of graded monoids. In this case, beyond minor terminology discrepancies, the factorability structures of~\cite{HeOz} correspond to what we call normalisations of class~$(4,3)$. But, in the general case, the two viewpoints are not directly comparable because of divergent treatment of units and invertible elements: in both approaches a ``dummy'' element is used, but with different assumptions, resulting in different notions of complexity and different conclusions. It seems that every factorability structure yields a normalisation of class~$(4,5)$, but understanding which normalisations of class~$(4,5)$ arise in this way remains open. 

Let us present our main results. The central technical notion is that of a \emph{normalisation}, which is a pair~$(\SS, \nm)$ made of a set~$\SS$ and an idempotent length-preserving map~$\nm$ from the free monoid~$\SS^*$ to itself: the intuition is that $\nm(\ww)$ is the result of normalising~$\ww$, that is, $\nm(\ww)$ is the distinguished element in the equivalence class of~$\ww$. The normalisation automatically determines the associated monoid via the defining relations~$\ww = \nm(\ww)$, and we take it as our basic object of investigation. We call \emph{quadratic} a normalisation~$(\SS, \nm)$ such that a word~$\ww$ is $\nm$-normal (meaning $\nm(\ww) = \ww$) if, and only if, each length-two factor of~$\ww$ is $\nm$-normal, and such that one can go from~$\ww$ to~$\nm(\ww)$ by applying a finite sequence of shifted copies of the restriction~$\nmr$ of~$\nm$ to the set~$\Pow\SS2$ of length-two words. We then introduce, for every quadratic normalisation, a \emph{class}, which is a pair of positive integers describing the complexity of normalisation for length-three words: by definition, if~$\ww$ is a length-three word, $\nm(\ww)$ is equal to $\nmr_{\!\aii\mm}(\ww)$ or $\nmr_{\!\ai\mm}(\ww)$, meaning a length-$\mm$ sequence of alternate applications of~$\nmr$ in positions~$1$-$2$ and~$2$-$3$, and we say that the class is~$(\cc, \cc')$ if one always reaches the normal form after at most~$\cc$ steps when starting from the left, and~$\cc'$ steps from the right. We observe that the class, if not infinite, has the form~$(\cc, \cc')$ with $\vert \cc' - \cc \vert \le 1$, and that a system of class~$(\cc, \cc')$ is of class~$(\dd, \dd')$ for all $\dd \ge \cc$ and~$\dd' \ge \cc'$. We give a number of examples witnessing possible behaviours for the class and its analogue for the normalisation of longer words. However, most of our general results involve quadratic normalisations of class~$(4, 3)$ or~$(3, 4)$.

The first main result is an axiomatisation of normalisations of class $(4, 3)$ in terms of the restriction of the normalisation map to length-two words:

\begin{theoA}
If $(\SS, \nm)$ is a quadratic normalisation of class~$(4, 3)$, then the restriction~$\nmr$ of~$\nm$ to~$\Pow\SS2$ is idempotent and satisfies $\nmr_{\!212} = \nmr_{\!2121} = \nmr_{\!1212}$. Conversely, if $\ft$ is an idempotent map on~$\Pow\SS2$ that satisfies $\ft_{212} = \ft_{2121} = \ft_{1212}$, there exists a quadratic normalisation~$(\SS, \nm)$ of class~$(4, 3)$ satisfying~$\ft=\nmr$.
\end{theoA}

The direct implication is easy and extends to all classes. But the converse direction is more delicate and does not extend: a map on length-two words normalising length-three words needs not normalise words of greater length. The proof of Theorem~A involves the monoid~$\MM_\pp$ studied in~\cite{KraArt} and~\cite{HeOz}, which is an asymmetric version of Artin's braid monoids where the relation $\sig2\sig1\sig2 = \sig1\sig2\sig1$ is replaced with $\sig2\sig1\sig2 = \sig1\sig2\sig1\sig2 = \sig2\sig1\sig2\sig1$. Let us mention that~\cite[Th.~3.4]{HeOz} is an analogue of Theorem~A for factorability structures.

The second main result involves termination. Every quadratic normalisation~$(\SS, \nm)$ gives rise to a quadratic rewriting system, namely the one with rules $\ww \to \nmr(\ww)$ for~$\ww$ a non-$\nm$-normal length-two word. By definition, this rewriting system is confluent and normalising, meaning that, for every initial word, there exists a finite sequence of rewriting steps leading to a unique $\nm$-normal word, but its convergence, meaning also that \emph{any} sequence of rewriting steps is finite, is a different question. We prove

\begin{theoB}
If $(\SS, \nm)$ is a quadratic normalisation of class~$(3,4)$ or~$(4,3)$, then the associated rewriting system is convergent, with every rewriting sequence starting from a length-$\pp$ word having length at most $2^\pp - \pp - 1$.
\end{theoB}

The result can be compared with the easier result that, in class~$(3, 3)$, every rewriting sequence starting from a length-$\pp$ word has length at most $\pp(\pp-1)/2$, and it is optimal, in the sense that there exists a nonconvergent rewriting system of class~$(4, 4)$. The proof of Theorem~B is delicate and relies on a diagrammatic tool called the domino rule. Theorem~B exhibits a strong difference between the factorability structures of~\cite{HeOz} and normalisations of class~$(4,3)$, since the former can induce nonterminating rewriting systems, as witnessed by the counter-example of~\cite[Appendix, Prop.~7]{HeOz}. However, there is a connection between Theorem~B and \cite[Th.~7.3]{HeOz}, which states termination in the case of a factorability structure that obeys the domino rule, hence, as a normalisation, is of class~$(4,3)$. The arguments are different, and it is not clear how restrictive it is for a normalisation of class~$(4,3)$ to be associated with a factorability structure.

As mentioned above, Garside normalisation~\cite{Garside} integrates into quadratic normalisations, more precisely normalisations of class~$(3, 3)$ in the case of a bounded Garside family, and of class~$(4, 3)$ in the general case. It is natural to ask for a characterisation of Garside systems inside the family of all normalisations of class~$(4, 3)$. This is the last one of our main results:

\begin{theoC}
Call a normalisation~$(\SS,\nm)$ left-weighted if, for all~$\ss,\tt$ in~$\SS$, the element~$\ss$ left-divides the first entry of~$\nm(\ss\sep\tt)$ in the associated monoid. Then, for every normalisation $(\SS, \nm)$ such that the associated monoid~$\MM$ is left-cancellative and contains no nontrivial invertible element, the family~$\SS$ is a Garside family in~$\MM$ and $(\SS, \nm)$ is the derived normalisation if, and only if, $(\SS, \nm)$ is of class~$(4, 3)$ and left-weighted.
\end{theoC}

The proof relies on nontrivial properties of Garside families and, again, on the domino rule available in class~$(4, 3)$. A consequence of Theorems~B and~C is that the rewriting system derived from a Garside family is always convergent, which generalises the case of Artin--Tits monoids with the elements of the corresponding Coxeter group as generators~\cite[Th~3.1.3,Prop.~3.2.1]{GaussentGuiraudMalbos}.

The paper is organised in five sections after this one. Section~\ref{S:Normal} contains basic definitions about normal forms and normalisations in the general case. We explain how the adjunction of a dummy generator with specific properties extends the use of length-preserving normalisations to non-graded monoids. In Section~\ref{S:Quadratic}, we introduce quadratic normalisations as those normalisations whose map is determined by its restriction to length-two words, and we establish a bijective correspondence between the latter and a generalisation of convergent rewriting systems (with termination relaxed into normalisation). We also introduce the class, and its generalisation the $\pp$-class, as  measures of the complexity of normalisation, and establish their basic properties. In particular, we give counterexamples showing the independence of the $3$-class and of the $\pp$-class for $\pp \ge 4$. Section~\ref{S:Class3} is devoted to the specific case of quadratic normalisations of class~$(4,3)$. Such systems provide well-behaved normalisation processes; we establish in particular an explicit universal formula for the normalisation of length-$\pp$ words and, as an application, we show that being of class~$(4,3)$ implies being of $\pp$-class~$(4,3)$ for every~$\pp$. The section ends with Theorem~A. In Section~\ref{S:Termination}, we study the relationship between the class of a quadratic normalisation and the termination of the associated rewriting system, proving in particular Theorem~B. Finally, Section~\ref{S:Garside} is devoted to the connection with Garside families and the associated greedy normal forms, establishing Theorem~C. 

Note that almost all observations in this paper extend from the context of monoids to that of categories, seen as monoids with a partially defined product.

%%%%
\subsection*{Thanks}
The authors warmly thank Viktoriya Ozornova for having sent a number of useful comments about the first version of this paper and shared her view of the subject.

%%%%
\section{Normalisations and geodesic normal forms}\label{S:Normal}

In this introductory section, we define normalisations and connect them with geodesic normal forms of monoids (Subsection~\ref{SS:Normal}). We explain how to add a ``dummy'' generator to make the restriction to length-preserving maps innocuous (Subsection~\ref{SS:NonGraded}). 

%%%%
\subsection{Normalisations}\label{SS:Normal}

If $\SS$ is a set, we denote by~$\SS^*$ the free monoid over~$\SS$ and call its elements \emph{$\SS$-words}, or simply \emph{words}. We write $\LG\ww$ for the length of an $\SS$-word~$\ww$, and $\ww \sep \ww'$, or simply~$\ww\ww'$, for the product of two $\SS$-words~$\ww$ and~$\ww'$.

Our aim is to investigate normal forms of a monoid~$\MM$ with respect to a generating family~$\SS$, that is, maps from~$\MM$ to~$\SS^*$ that choose, for every element~$\gg$ of~$\MM$, a distinguished expression of~$\gg$ by an $\SS$-word, or, equivalently, maps from~$\SS^*$ to itself that choose a distinguished element in each equivalence class. We shall privilege the latter approach, in which the primary object is the word map and the monoid is then derived from it. 

\begin{defi}\label{D:NormSys}
A \emph{normalisation} is a pair $(\SS,\nm)$, where~$\SS$ is a set and~$\nm$ is a map from~$\SS^*$ to itself satisfying, for all $\SS$-words~$\uu, \vv, \ww$,
\begin{gather}
\label{E:NormSys1}
\LG{\nm(\ww)} = \LG\ww, \\
\label{E:NormSys2}
\LG\ww = 1 \text{\ implies\ } \nm(\ww) = \ww, \\
\label{E:NormSys3}
\nm(\uu \sep \nm(\ww) \sep \vv)=\nm(\uu \sep \ww \sep \vv).
\end{gather}
An $\SS$-word~$\ww$ satisfying $\nm(\ww)=\ww$ is called \emph{$\nm$-normal}. If~$\MM$ is a monoid, we say that $(\SS, \nm)$ is a normalisation \emph{for~$\MM$} if~$\MM$ admits the presentation 
\begin{equation}\label{E:NormSys4}
\PRESp{\SS}{\{\ww = \nm(\ww) \mid \ww \in \SS^*\}}.
\end{equation}
\end{defi}

Note that~\eqref{E:NormSys3} implies that~$\nm$ is idempotent. The homogeneity condition~\eqref{E:NormSys1} is discussed (and partly skirted around) in Subsection~\ref{SS:NonGraded}.

\begin{exam}\label{X:Abelian}
Assume that $\SS$ is a set and $<$ is a linear order on~$\SS$. For~$\ww$ in~$\SS^*$, define~$\nm(\ww)$ to be the $\lLex$-minimal word obtained by permuting letters in~$\ww$, where~$\lLex$ is the lexicographic extension of~$<$ to~$\SS^*$. So, for instance, assuming $\tta, \ttb, \ttc \in \SS$ and $\tta < \ttb < \ttc$, we find $\nm(\ttb\ttc\tta\ttb\tta\ttc) = \tta\tta\ttb\ttb\ttc\ttc$. Then $(\SS, \nm)$ is a normalisation for the free commutative monoid~$\NNNN^{(\SS)}$ over~$\SS$. 
\end{exam}

The following fact is a direct consequence of the definition:

\begin{lemm}\label{L:Graded}
If $(\SS, \nm)$ is a normalisation for a monoid~$\MM$, then~$\MM$ admits a graduation such that all elements of~$\SS$ have degree one, that is, there exists a morphism~$\dd : \MM \to (\NNNN, +)$ such that $\ss \in \SS$ implies $\dd(\ss) = 1$.
\end{lemm}

\begin{proof}
For~$\gg$ in~$\MM$, all the $\SS$-words representing~$\gg$ must have the same length by~\eqref{E:NormSys1}: define $\dd(\gg)$ to be this common length. 
\end{proof}

The following result connects Definition~\ref{D:NormSys} with the alternative approach in which the monoid is given first. If a monoid~$\MM$ is generated by a set~$\SS$, we denote by~$\ev$ the canonical projection from~$\SS^*$ to~$\MM$. 

\begin{lemm}\label{L:NormSysCharacterisation}
Assume that~$\MM$ is a monoid and~$\SS$ is a generating subfamily of~$\MM$. If~$\nm$ is a length-preserving map from~$\SS^*$ to itself, then $(\SS,\nm)$ is a normalisation for~$\MM$ if, and only if, for all $\SS$-words~$\ww, \ww'$, the following conditions hold:
\begin{gather}
\label{E:NormMap1}
\ev(\nm(\ww)) = \ev(\ww),\\
\label{E:NormMap2}
\ev(\ww) = \ev(\ww') \text{ implies } \nm(\ww) = \nm(\ww').
\end{gather}
\end{lemm}

\begin{proof}
Assume that $(\SS, \nm)$ is a normalisation for~$\MM$. As~\eqref{E:NormSys4} is a presentation of~$\MM$, each relation $\nm(\ww) = \ww$ is valid in~$\MM$ and, therefore, \eqref{E:NormMap1} holds. Next, assume that $\ww, \ww'$ are $\SS$-words satisfying $\ev(\ww) = \ev(\ww')$. As~\eqref{E:NormSys4} is a presentation of~$\MM$, \eqref{E:NormMap2} follows from $\nm(\uu\sep\ww\sep\vv)=\nm(\uu\sep\nm(\ww)\sep\vv)$, which is~\eqref{E:NormSys3}.

Conversely, assume that~\eqref{E:NormMap1} and~\eqref{E:NormMap2} are satisfied. By assumption on~$\nm$, \eqref{E:NormSys1} is satisfied and, for $\ss$ in~$\SS$, we have $\nm(\ss)\in\SS$, so that $\SS\subseteq\MM$ implies $\ev(\ss) = \ss$ and $\ev(\nm(\ss)) = \nm(\ss)$, whence~\eqref{E:NormSys2} by~\eqref{E:NormMap1}. Then, for $\SS$-words $\uu, \vv, \ww$, we have $\ev(\uu\sep\nm(\ww)\sep\vv) = \ev(\uu\sep\ww\sep\vv)$ by~\eqref{E:NormMap1}, whence~\eqref{E:NormSys3} by~\eqref{E:NormMap2}. So~$(\SS, \nm)$ is a normalisation. Finally, \eqref{E:NormSys4} is a presentation of~$\MM$ because, on the one hand, all relations $\nm(\ww) = \ww$ are valid in~$\MM$ by~\eqref{E:NormMap1} and, on the other hand, $\ev(\ww) = \ev(\ww')$ implies $\nm(\ww) = \nm(\ww')$ by~\eqref{E:NormMap2}, hence~\eqref{E:NormMap1} implies that~$\ww$ and~$\ww'$ are equivalent to~$\nm(\ww)$ modulo the relations of~\eqref{E:NormSys4}.
\end{proof}

We now connect normalisations with the usual notion of a normal form.

\begin{defi}\label{D:NF}
If $\MM$ is a monoid and $\SS$ is a generating subfamily of~$\MM$, a \emph{normal form on $(\MM,\SS)$} is a (set-theoretic) section of the canonical projection~$\ev$ of~$\SS^*$ onto~$\MM$. A normal form~$\NF$ on $(\MM,\SS)$ is called \emph{geodesic} if, for every~$\gg$ in~$\MM$, we have $\LG{\NF(\gg)} \le \LG\ww$ for every $\SS$-word~$\ww$ representing~$\gg$. 
\end{defi}

For graded monoids, normalisations are equivalent to normal forms:

\begin{prop}\label{P:NF}
\ITEM1 If $(\SS,\nm)$ is a normalisation for a monoid~$\MM$, we obtain a normal form on~$(\MM,\SS)$ by putting 
\begin{equation}\label{E:NF1}
\NF(\gg) = \nm(\ww), \quad\text{where $\ww$ is any representative of~$\gg$}.
\end{equation}

\ITEM2 Conversely, assume that $\MM$ is a graded monoid, $\SS$ is a generating subfamily of~$\MM$ whose elements have degree~$1$, and $\NF$ is normal form on~$(\MM,\SS)$. Then we obtain a normalisation~$(S,\nm)$ for~$\MM$ by putting 
\begin{equation}\label{E:NF2}
\nm(\ww) = \NF(\ev(\ww)).
\end{equation}

\ITEM3 The correspondences of~\ITEM1 and~\ITEM2 are inverses of one another.
\end{prop}

\begin{proof}
\ITEM1 First the definition makes sense, since, if $\ww, \ww'$ are two representatives of~$\gg$, then \eqref{E:NormMap2} implies $\nm(\ww) = \nm(\ww')$. Next, assuming $\ev(\ww) = \gg$, we obtain $\ev(\NF(\gg)) = \ev(\nm(\ww)) = \ev(\ww) = \gg$ using~\eqref{E:NormMap1}, so~$\NF$ is a section of~$\ev$. 

\ITEM2 The assumption that~$\MM$ is graded implies $\LG{\nm(\ww)} = \LG{\ww}$ for every~$\SS$-word~$\ww$. Then, \eqref{E:NF2} implies $\ev(\nm(\ww)) = \ev(\NF(\ev(\ww))) = \ev(\ww)$ because $\NF$ is a section of~$\ev$, so~\eqref{E:NormMap1} holds. Finally, $\ev(\ww)=\ev(\ww')$ implies $\NF(\ev(\ww))=\NF(\ev(\ww'))$, whence~\eqref{E:NormMap2}. So, by Lemma~\ref{L:NormSysCharacterisation}, $(\SS,\nm)$ is a normalisation for~$\MM$.

\ITEM3 If~$(\SS,\nm)$ is a normalisation for~$\MM$, and~$\NF$ is defined by~\eqref{E:NF1} and~$\nm'$ by~\eqref{E:NF2}, then $\nm'(\ww) = \NF(\ev(\ww)) = \nm(\ww)$ holds, since~$\ww$ is a representative of~$\ev(\ww)$. Conversely, if~$\NF$ is a normal form on~$(\MM,\SS)$, and~$\nm$ is defined by~\eqref{E:NF2} and~$\NF'$ by~\eqref{E:NF1}, then $\ev(\ww) = \gg$ implies $\NF'(\gg) = \nm(\ww) = \NF(\gg)$. Hence the correspondences of~\ITEM1 and~\ITEM2 are inverses of one another.
\end{proof}

%%%%
\subsection{The non-graded case}\label{SS:NonGraded}

So far, according to Lemma~\ref{L:Graded}, only graded monoids are eligible. We explain how to adapt our approach to arbitrary monoids. 

\begin{defi}\label{D:Neutral}
If $(\SS, \nm)$ is a normalisation, an element~$\ee$ of~$\SS$ is called \emph{$\nm$-neutral} if 
\begin{equation}\label{E:Neutral1}
\nm(\ww \sep \ee) = \nm(\ee \sep \ww) = \nm(\ww) \sep \ee
\end{equation}
hold for every~$\SS$-word~$\ww$. If~$\MM$ is a monoid, we say that $(\SS, \nm)$ is a normalisation \emph{mod~$\ee$} for~$\MM$ if $\ee$ is an $\nm$-neutral element of~$\SS$ and $\MM$ admits the presentation 
\begin{equation}\label{E:Neutral2}
\PRESp{\SS}{\{\ww = \nm(\ww) \mid \ww \in \SS^*\} \cup \{\ee = 1\}}.
\end{equation}
We then put $\SSm = \SS \setminus \{\ee\}$, and write~$\ev_{\ee}$ for the canonical projection of~$\SSm^*$ onto~$\MM$.
\end{defi}

If $(\SS,\nm)$ is a normalisation for a monoid~$\MM$, \eqref{E:Neutral1} implies that there exists at most one $\nm$-neutral element in~$\SS$. Then, if~$\ee$ is such an element, $(\SS, \nm)$ is a normalisation mod~$\ee$ for the monoid obtained by collapsing~$\ee$ in~$\MM$.

\begin{lemm}\label{L:Neutral}
Assume that $(\SS, \nm)$ is a normalisation mod~$\ee$ for a monoid~$\MM$, and let~$\pi_\ee$ be the canonical projection from~$\SS^*$ onto~$\SSm^*$.

\ITEM1 The monoid~$\MM$ admits the presentation
\begin{equation}\label{E:Neutral3}
\PRESp{\SSm}{\{\ww = \pi_\ee(\nm(\ww)) \mid \ww \in \SSm^*\}}.
\end{equation}

\ITEM2 For all $\SS$-words $\ww_0 \wdots \ww_\ell$, we have
\begin{equation}\label{E:Neutral4}
\nm(\ww_0 \sep \ee \sep \ww_1 \sep \pdots \sep \ww_{\ell-1} \sep \ee \sep \ww_\ell) = \nm(\ww_0 \sep \pdots \sep \ww_\ell) \sep \ee^\ell.
\end{equation}

\ITEM3 For every $\SS$-word~$\ww$, we have $\nm(\ww) = \ww' \sep \ee^{\ell}$, where~$\ww'$ is an $\SS_{\ee}$-word and~$\ell$ is an upper bound of the number of occurrences of~$\ee$ in~$\ww$. 
\end{lemm}

\begin{proof}
\ITEM1 By definition, $\MM$ is generated by~$\SS$, hence by~$\SS_\ee$. Next, for every $\SSm$-word~$\ww$, the relation $\ww = \pi_{\ee}(\nm(\ww))$ is valid in~$\MM$ owing to $\pi_{\ee}(\ww)=\ww$. As~$\MM$ admits the presentation~\eqref{E:Neutral2}, it remains to check that all relations of~\eqref{E:Neutral2} can be derived from those of~\eqref{E:Neutral3} plus $\ee = 1$: this is because the $\SS$-word~$\nm(\ww)$ is obtained from~$\pi_\ee(\nm(\ww))$ by inserting copies of~$\ee$ thanks to the relation $\ee = 1$.

\ITEM2 Use induction on~$\ell\ge 0$. For~$\ell = 0$, the result is immediate. For~$\ell \ge 1$, we~find
\begin{align*}
\nm(\ww_0 \sep \ee \sep \ww_1 \sep \pdots \sep \ww_{\ell-1} \sep \ee \sep \ww_\ell)
&= \nm(\ww_0 \sep \ee \sep \ww_1 \sep \pdots \sep \ww_{\ell-1} \sep \nm(\ee \sep \ww_\ell))
&&\text{by~\eqref{E:NormSys3}}\\
&= \nm(\ww_0 \sep \ee \sep \ww_1 \sep \pdots \sep \ww_{\ell-1} \sep \ww_\ell \sep \ee)
&&\text{by~\eqref{E:Neutral1}}\\
&= \nm(\ww_0 \sep \ee \sep \ww_1 \sep \pdots \sep \ww_{\ell-1} \sep \ww_\ell) \sep \ee
&&\text{by~\eqref{E:Neutral1}}\\
&= \nm(\ww_0 \sep \pdots \sep \ww_{\ell-1} \sep \ww_\ell) \sep \ee^{\ell}
&&\text{by induction hypothesis.}
\end{align*}

\ITEM3 By~\ITEM2, we have $\nm(\ww)=\vv|\ee^{\pp}$, where $\vv$ is $\nm(\pi_{\ee}(\ww))$ and~$\pp$ is the number of occurrences of~$\ee$ in~$\ww$. By the same argument, we find that $\nm(\vv)$ is $\ww'\sep\ee^{\qq}$, where~$\ww'$ is $\nm(\pi_{\ee}(\vv))$ and~$\qq$ is the number of occurrences of~$\ee$ in~$\vv$. Since~$\nm$ is idempotent, we have $\vv=\nm(\vv)=\ww'\sep\ee^{\qq}$. We deduce that~$\ww'$ contains no~$\ee$ and that $\nm(\ww)=\ww'\sep\ee^{\pp+\qq}$ holds. \end{proof}

The following variation of Proposition~\ref{P:NF} requires no graduation assumption: the solution is to add a dummy generator to preserve word length.

\begin{prop}\label{P:GenNF}
\ITEM1 If $(\SS,\nm)$ is a normalisation mod~$\ee$ for a monoid~$\MM$, we obtain a geodesic normal form on~$(\MM,\SSm)$ by putting
\begin{equation}\label{E:GenNF1}
\NF(\gg) = \pi_\ee(\nm(\ww)), \quad\text{where $\ww$ is any representative of~$\gg$ in~$\SS^*$}. 
\end{equation}

\ITEM2 Conversely, assume that~$\MM$ is a monoid, $\SS$ is a generating subfamily of~$\MM$, and~$\NF$ is a geodesic normal form on~$(\MM,\SS)$. Put $\SSp = \SS \amalg \{\ee\}$ and write~$\ev^{\ee}$ for the canonical projection of~$\SS^*$ onto~$\MM$ extended to~$(\SSp)^*$ by $\ev^\ee(\ee) = 1$. Then we obtain a normalisation~$(\SSp,\nm)$ mod~$\ee$ for~$\MM$ by putting 
\begin{equation}\label{E:GenNF2}
\nm(\ww) = \NF(\ev^{\ee}(\ww))  \sep \ee^\mm, \quad \text{with $\mm = \LG\ww - \LG{\NF(\ev^{\ee}(\ww))}$.}
\end{equation}

\ITEM3 The correspondences of~\ITEM1 and~\ITEM2 are inverses of one another.
\end{prop}

\begin{proof}
\ITEM1 Let $\ww,\ww'$ be two representatives of~$\gg$ in~$\SS^*$. Then one can be obtained from the other by applying relations $\vv=\vv'$ with either $\vv=\uu_1\sep\nm(\uu_2)\sep\uu_3$ and $\vv'=\uu_1\sep\uu_2\sep\uu_3$, or $\vv=\uu_1\sep\ee\sep\uu_2$ and $\vv'=\uu_1\sep\uu_2$. In the first case, \eqref{E:NormSys3} gives~$\nm(\vv)=\nm(\vv')$. In the second case, Lemma~\ref{L:Neutral}~\ITEM2 gives $\nm(\vv)=\nm(\vv')\sep\ee$. Thus we have $\pi_{\ee}(\nm(\ww))=\pi_{\ee}(\nm(\ww'))$ and $\NF_{\ee}(\gg)$ is well defined. 

Now, let~$\MM^{\ee}$ be the monoid presented by~\eqref{E:NormSys4}, $\pi:\MM^{\ee}\twoheadrightarrow\MM$ and~$\ev_{\ee}:\SSm^*\twoheadrightarrow\MM$ be the canonical projections. For~$\gg$ in~$\MM$ and~$\ww$ be a representative of~$\gg$ in~$\SS^*$, the relation $\pi\comp\ev = \ev_{\ee}\comp\pi_{\ee}$ and~\eqref{E:NormMap1} imply 
$$\ev_{\ee}(\NF(\gg)) = \ev_{\ee}(\pi_{\ee}(\nm(\ww)) = \pi(\ev(\nm(\ww))) = \pi(\ev(\ww)) = \gg,$$
so $\NF$ is a normal form on $(\MM,\SS_{\ee})$. Moreover, we have $\LG{\pi_{\ee}(\nm(\ww))} \le \LG{\nm(\ww)} = \LG{\ww}$, so $\NF$ is geodesic.

\ITEM2 Since~$\NF$ is geodesic, we have $\LG{\NF(\ev^{\ee}(\ww))} \le \LG\ww$ for every $\SSp$-word~$\ww$. So~\eqref{E:GenNF2} makes sense and~$\nm$ is length-preserving. Next, for $\ss\in\SSp$, we have either $\ss\in\SS$ and $\nm(\ss)=\NF(\ss)=\ss$, or $\ss=\ee$ and  $\nm(\ee) = {\NF(1)\sep\ee=\ee}$. Then, since~$\NF$ is a section of~$\ev$, \eqref{E:GenNF2} gives $\ev^{\ee}(\nm(\ww)) = \ev(\NF(\ev^{\ee}(\ww)) = \ev^{\ee}(\ww)$, yielding $\NF(\ev^{\ee}(\uu\sep\nm(\ww)\sep\vv)) = \NF(\ev^{\ee}(\uu\sep\ww\sep\vv))$ and, since~$\nm$ is length-preserving, $\nm(\uu\sep\nm(\ww)\sep\vv)=\nm(\uu\sep\ww\sep\vv)$. Thus $(\SSp,\nm)$ is a normalisation. 

Now, let~$\ww$ be an $\SSp$-word, with $\mm = \LG\ww - \LG{\NF(\ev^{\ee}(\ww))}$. Then we have $\ev^{\ee}(\ww \sep \ee) = \ev^{\ee}(\ee \sep \ww) = \ev^{\ee}(\ww)$, whence $\nm(\ww \sep \ee) =\NF(\ev^{\ee}(\ww)) \sep \ee^{\mm+1} = \nm(\ww) \sep \ee$, and, similarly, $\nm(\ee \sep \ww) = \nm(\ww) \sep \ee$, so $\ee$ is $\nm$-neutral. Finally, by Lemma~\ref{L:Neutral}\ITEM1, the monoid~$\MM$ admits the presentation $\PRESp{(\SSp)_{\ee}}{\{\ww = \pi_{\ee}(\nm(\ww)) \mid \ww \in (\SSp)_{\ee}^*\}}$. Owing to the equalities $(\SSp)_{\ee}=\SS$ and $\pi_{\ee}(\nm(\ww))=\pi_{\ee}(\NF(\ev^{\ee}(\ww)) \sep \ee^m)=\NF(\ev(\ww))$, the monoid~$\MM$ also admits the presentation $\PRESp{\SS}{\{\ww=\NF(\ev(\ww)) \mid \ww\in\SS^*\}}$. 

\ITEM3 Starting from~\ITEM1, let $(\SSp, \nm')$ be the normalisation mod~$\ee$ derived from~$\NF$ using~\ITEM2. Then we have $(\SS_{\ee})^{\ee}=\SS$ and $\nm'(\ww) = \NF(\ev(\ww)) \sep \ee^m = \pi_{\ee}(\nm(\ww)) \sep \ee^m$, whence $\nm'(\ww)=\nm(\ww)$ by Lemma~\ref{L:Neutral}~\ITEM3. Conversely, starting from~\ITEM2, let $\NF'$ be the normal form derived from~$\nm$ using~\ITEM1. Then we have $\NF'(\gg) = \pi_{\ee}(\nm(\ww)) = \pi_{\ee}(\NF(\ev^{\ee}(\ww))\sep\ee^{\mm}) = \NF(\gg)$ for every~$\gg$ in~$\MM$ with representative $\SS$-word~$\ww$.
\end{proof}

\begin{rema*}
If a monoid~$\MM$ is graded with respect to a generating family~$\SS$ and~$\NF$ is a normal form on~$(\MM, \SS)$, then two normalisations come associated with~$\MM$ and~$\SS$: the one $(\SS, \nm)$ provided by Proposition~\ref{P:NF}~\ITEM2, and the one $(\SSp,\nm^\ee)$ provided by Proposition~\ref{P:GenNF}~\ITEM2. The connection between these systems is given, for every $\SSp$-word~$\ww$, by the equality $\pi_{\ee}(\nm^\ee(\ww)) = \nm(\pi_{\ee}(\ww))$.
\end{rema*}

%%%%
\section{Quadratic normalisations and their class}\label{S:Quadratic}

We now restrict our study to particular normalisations that are, in a convenient sense, generated by transformations of length-two words. After basic definitions and examples (Subsection~\ref{SS:Quadratic}), we relate those normalisations with rewriting systems (Subsection~\ref{SS:Rewrite}). Then we introduce the class of such a normalisation as a pair of elements of~$\NNNN \cup \{\infty\}$ that gives an upper bound on the complexity of normalisation for length-three words (Subsection~\ref{SS:Class}). Finally, we consider the $\pp$-class, an analogue involving length-$\pp$ words (Subsection~\ref{SS:pClass}).

%%%%
\subsection{Quadratic normalisations}\label{SS:Quadratic}

\begin{nota}\label{N:Local}
\ITEM1 If~$\SS$ is a set and~$\ft$ is a map from the set~$\Pow\SS\pp$ of length-$\pp$ $\SS$-words to itself, then, for $\ii\ge 1$, we denote by~$\ft_\ii$ the (partial) map of~$\SS^*$ to itself that consists in applying~$\ft$ to the entries in position~$\ii\wdots\ii+p-1$. If $\uu= \ii_1 \sep \pdots \sep \ii_{\nn}$ is a finite sequence of positive integers, we write~$\ft_{\uu}$ for the composite map $\ft_{\ii_{\nn}}\comp \pdots \comp \ft_{\ii_1}$. 

\ITEM2 If $(\SS, \nm)$ is a normalisation, we denote by~$\nmr$ the restriction of~$\nm$ to~$\Pow\SS2$.
\end{nota}

Here is the main notion investigated in this paper:

\begin{defi}\label{D:Quad}
A normalisation $(\SS,\nm)$ is called \emph{quadratic} if the following conditions hold:
\begin{gather}\label{E:Quad1}
\parbox[c]{130mm}{An $\SS$-word~$\ww$ is $\nm$-normal if, and only if, every length-two factor of~$\ww$ is.}\\ 
\label{E:Quad2}
\parbox[c]{130mm}{For every $\SS$-word~$\ww$, there exists a finite sequence~$\uu$ of positions, depending on~$\ww$, such that $\nm(\ww)$ \VR(3,0) is equal to~$\nmr_{\!\uu}(\ww)$.}
\end{gather}
\end{defi}

So, a normalisation~$(\SS, \nm)$ is quadratic if $\nm$-normality only depends on length-two factors and if one can go from an $\SS$-word~$\ww$ to the $\SS$-word~$\nm(\ww)$ in finitely many steps, each of which consists in applying~$\nm$ to some length-two factor. Note that, provided $\SS$ is finite, \eqref{E:Quad1} implies that the language of all $\nm$-normal $\SS$-words is regular.

\begin{exam}
The normalisation~$(\SS, \nm)$ of Example~\ref{X:Abelian} is quadratic. Indeed, an $\SS$-word is $\nm$-normal if, and only if, all its length-two subfactors are of the form~$\ss\sep\tt$ with $\ss \le \tt$, so~\eqref{E:Quad1} is satisfied. Moreover, \eqref{E:Quad2} holds, since every $\SS$-word~$\ww$ can be transformed into the equivalent $\nm$-normal $\SS$-word~$\nm(\ww)$ by switching adjacent letters that are not in the expected order: for instance, if $\tta < \ttb < \ttc$, one has $\nm( \ttc\ttb\ttb\tta) = \tta\ttb\ttb\ttc = \nmr_{\!31213}(\ttc\ttb\ttb\tta)$. Note that the sequence of length-two normalisations is not unique, and depends on the initial word.
\end{exam}

Definition~\ref{D:Quad} gathers two locality conditions, which, taken separately, do not seem to have interesting consequences in our approach: \eqref{E:Quad1} is a static characterisation of normal words, whereas~\eqref{E:Quad2} is dynamical in that it involves transformations into normal words. As~\eqref{E:Quad2} implies that a length-two word is $\nm$-normal if, and only if, it is $\nmr$-invariant, it induces the right-to-left implication in~\eqref{E:Quad1}. The next two counterexamples show that this is the only general connection between~\eqref{E:Quad1} and~\eqref{E:Quad2}.

\begin{exam}\label{X:LocNotQuad}
Let $\SS=\{\tta,\ttb,\ttc\}$ and~$\nm:\SS^*\fl\SS^*$ be defined as follows: starting from~$\ww$ in~$\SS^*$, we first replace every factor $\tta\ttb\tta$ or $\tta\ttc\tta$ with~$\tta^3$, and then, in the resulting word, we replace every factor~$\tta\ttb$ with~$\tta\ttc$ and every factor~$\ttc\tta$ with~$\ttb\tta$. Then $(\SS, \nm)$ is a normalisation (for~\eqref{E:NormSys3}, observe for instance that $\nm(\uu\tta\ttb\vv) = \nm(\uu\tta\ttc\vv)$ holds both when $\vv$ begins with~$\tta$ and when it does not), it satisfies~\eqref{E:Quad1} since a word is $\nm$-normal if and only if it contains no factor~$\tta\ttb$ or~$\ttc\tta$, but it does not satisfy~\eqref{E:Quad2}: neither $\tta\ttb\tta$ nor $\tta\ttc\tta$ is $\nm$-normal, but the only $\SS$-words that can be obtained from~$\tta\ttb\tta$ and~$\tta\ttc\tta$ using~$\nmr_{\!1}$ and $\nmr_{\!2}$ are~$\tta\ttb\tta$ and~$\tta\ttc\tta$ themselves. Hence \eqref{E:Quad1} does not imply~\eqref{E:Quad2}.
\end{exam}

\begin{exam}\label{X:QuadNotLoc}
Let $\SS = \{\tta, \ttb\}$ and $\nm : \SS^* \to \SS^*$ be defined by $\nm(\ww) = \ww$ for $\Vert\ww\Vert \le 1$, and $\nm(\ww) = \tta^{\Vert\ww\Vert-1}\ttb$ for $\Vert\ww\Vert \ge 2$. Then $(\SS, \nm)$ is a normalisation for the monoid~$\PRESp\SS{\tta\ttb = \ttb\tta = \ttb^2 = \tta^2}$. Now~\eqref{E:Quad1} fails, since $\tta\tta\ttb$, which is $\nm$-normal, contains the non $\nm$-normal factor~$\tta\tta$. But~\eqref{E:Quad2} is satisfied, since a straightforward induction gives $\nm(\ww) = \nmr_ {1\sep\pdots \sep\pp-1} (\ww)$ for~$\ww$ of length~$\pp \ge 2$. Hence~\eqref{E:Quad2} does not imply~\eqref{E:Quad1}.
\end{exam}

When a normalisation~$(\SS, \nm)$ is quadratic, the restriction~$\nmr$ of~$\nm$ to~$\Pow\SS2$ is crucial. Here are first general properties.

\begin{prop}\label{P:QuadPres}
\ITEM1 If $(\SS,\nm)$ is a quadratic normalisation for a monoid~$\MM$, then~$\nmr$ is idempotent and~$\MM$ admits the presentation 
\begin{equation}\label{E:QuadPres}
\PRESp{\SS}{\{\ss\sep\tt = \nmr(\ss\sep\tt) \mid \ss, \tt \in \SS\}};
\end{equation}

\ITEM2 If $(\SS,\nm)$ is a quadratic normalisation, then an element~$\ee$ of~$\SS$ is $\nm$-neutral if, and only if, it satisfies
\begin{equation}\label{E:QuadNeutral1}
\nmr(\ee\sep\ss) = \nmr(\ss\sep\ee) = \ss\sep\ee \qquad \text{for every~$\ss$ in~$\SS$.}
\end{equation}

\ITEM3 If $(\SS,\nm)$ is a quadratic normalisation mod~$\ee$ for a monoid~$\MM$, then~$\MM$ admits the presentation
\begin{equation}\label{E:QuadNeutral2}
\PRESp{\SSm}{\{\ss\sep\tt = \pi_{\ee} (\nmr(\ss\sep\tt)) \mid \ss, \tt \in \SSm\}}. 
\end{equation}
\end{prop}

\begin{proof}
\ITEM1 By~\eqref{E:NormSys3}, $\nm$ is idempotent, hence so is its restriction~$\nmr$. The monoid~$\MM$ admits the presentation~\eqref{E:QuadPres} because it admits the presentation~\eqref{E:NormSys4} with the same generators, because the relations~\eqref{E:QuadPres} are contained into the ones of~\eqref{E:NormSys4}, and because~\eqref{E:Quad2} implies that every relation of~\eqref{E:NormSys4} is a consequence of finitely many relations of~\eqref{E:Quad2}. 

\ITEM2 The relations~\eqref{E:QuadNeutral1} are particular instances of~\eqref{E:Neutral1}, so they hold if~$\ee$ is $\nm$-neutral. Conversely, assume~\eqref{E:QuadNeutral1} and let~$\ww$ be an $\SS$-word. 

We first prove $\nm(\ww \sep \ee)=\nm(\ww) \sep \ee$. We have $\nm(\ww \sep \ee)=\nm(\nm(\ww) \sep \ee)$ by~\eqref{E:NormSys3}. Moreover, every length-two factor of $\nm(\ww) \sep \ee$ is $\nm$-normal. Indeed, for $\LG{\ww}\ge 1$, writing $\nm(\ww)=\ww' \sep \ss$ with $\ww'\in\SS^*$ and $\ss\in\SS$, the length-two factors of $\nm(\ww) \sep \ee$ are those of $\nm(\ww)$, which are $\nm$-normal by~\eqref{E:Quad1}, and $s \sep \ee$, which is $\nm$-normal by~\eqref{E:QuadNeutral1}. Hence $\nm(\ww) \sep \ee$ is $\nm$-normal by~\eqref{E:Quad1}, which implies $\nm(\ww \sep \ee)=\nm(\ww) \sep \ee$.

Now, we prove $\nm(\ee \sep \ww)=\nm(\ww \sep \ee)$ by induction on~$\LG\ww$. The result is immediate for $\LG\ww = 0$. Otherwise, write $\ww=\ss \sep \ww'$, with $\ss$ in~$\SS$ and $\ww'$ satisfying~$\nm(\ee \sep \ww')=\nm(\ww' \sep \ee)$. Using~\eqref{E:NormSys3},~\eqref{E:QuadNeutral1} and the induction hypothesis on~$\ww'$, we find $
\nm(\ee \sep \ss \sep \ww')
	= \nm(\nm(\ee \sep \ss) \sep \ww')
	= \nm(\ss \sep \ee \sep \ww')
	= \nm(\ss \sep \nm(\ee \sep \ww'))
	= \nm(\ss \sep \nm(\ww' \sep \ee))
	= \nm(\ss \sep \ww' \sep \ee)
$, so $\ee$ is $\nm$-neutral. 

\ITEM3 By~\ITEM1, $\MM$ is presented by $\PRESp{\SS}{\{\ss\sep\tt = \nmr(\ss\sep\tt) \mid \ss, \tt \in \SS\}\cup\{\ee=1\}}$. Applying the Tietze transformation that collapses~$\ee$ onto~$1$, we obtain the presentation $\PRESp{\SS_{\ee}}{\{\pi_{\ee}(\ss\sep\tt) = \pi_{\ee}(\nmr(\ss\sep\tt)) \mid \ss, \tt \in \SS\}}$ for~$\MM$. If at least one of~$\ss$ or~$\tt$ is~$\ee$, then by~\eqref{E:QuadNeutral1}, the corresponding relation boils down to $\ss=\ss$, $\tt=\tt$ or $1=1$, so that we can remove it. Otherwise, we have~$\pi_{\ee}(s \sep t)=s \sep t$, yielding~\eqref{E:QuadNeutral2}.
\end{proof}

%%%%
\subsection{Quadratic normalisations and rewriting}\label{SS:Rewrite}

We recall that a \emph{(word) rewriting system} is a pair $(\SS,\RR)$ consisting of a set~$\SS$ and a binary relation~$\RR$ on~$\SS^*$ whose elements $(\ww,\ww')$ are written $\ww\fl\ww'$ and called \emph{rewriting rules}. 

Assume that~$(\SS,\RR)$ is a rewriting system. We denote by~$\fl_{\RR}$ the closure of~$\RR$ with respect to the product of~$\SS^*$ and by~$\fl_{\RR}^*$ the reflexive-transitive closure of~$\fl_{\RR}$. An $\SS$-word~$\ww$ is \emph{$\RR$-normal} if~$\ww\fl_{\RR}^*\ww'$ implies $\ww'=\ww$. If $\ww,\ww'$ are $\SS$-words, $\ww'$ is an \emph{$\RR$-normal form of~$\ww$} if~$\ww\fl_{\RR}^*\ww'$ and~$\ww'$ is $\RR$-normal. One says that $(\SS,\RR)$ is \emph{quadratic} if $\ww\fl\ww'\in\RR$ implies $\LG{\ww}=\LG{\ww'} = 2$; \emph{reduced} if $\ww\fl\ww'\in\RR$ implies that~$\ww'$ is $\RR$-normal and $\ww$ is $\RR\setminus\{\ww\fl\ww'\}$-normal; \emph{normalising} if every $\SS$-word admits at least an $\RR$-normal form; and \emph{confluent} if the conjunction of $\ww\fl_{\RR}^*\ww_1$ and $\ww\fl_{\RR}^*\ww_2$ implies $\ww_1\fl_{\RR}^*\ww'$ and $\ww_2\fl_{\RR}^*\ww'$ for some~$\ww'$. As a rewriting rule is a pair of words, there is no ambiguity in speaking of the monoid \emph{presented by~$(\SS, \RR)$}. 

\begin{prop}
\label{P:QuadNormRewr}
\ITEM1 If~$(\SS,\nm)$ is a quadratic normalisation for a monoid~$\MM$, then we obtain a quadratic, reduced, normalising and confluent rewriting system~$(\SS, \RR)$ presenting~$\MM$ by putting 
\begin{equation}
\label{E:QuadRewrOfNorm}
\RR = \{\ss\sep\tt \fl \nmr(\ss\sep\tt) \mid \ss,\tt\in\SS, \ss\sep\tt\neq\nmr(\ss\sep\tt)\}.
\end{equation}

\ITEM2 Conversely, if $(\SS,\RR)$ is a quadratic, reduced, normalising and confluent rewriting system presenting a monoid~$\MM$, we obtain a quadratic normalisation~$(\SS, \nm)$ for~$\MM$ by putting 
\begin{equation}
\label{E:NormOfRewr}
\nm(\ww) = \ww' \quad \text{where~$\ww'$ is the $\RR$-normal form of~$\ww$.}
\end{equation}

\ITEM3 The correspondences of \ITEM1 and \ITEM2 are inverses of one another.
\end{prop}

\begin{proof}
\ITEM1 By definition, $(\SS,\RR)$ is quadratic and reduced, and, for all $\SS$-words~$\ww$ and~$\ww'$, we have $\ww\fl_{\RR}^*\ww'$ if, and only if, $\ww'=\nmr_{\!\uu}(\ww)$ holds for some sequence~$\uu$ of positions. Thus, by~\eqref{E:Quad2}, $\RR$ is normalising. Moreover, the conjunction of $\ww\fl_{\RR}^*\ww_1$ and $\ww\fl_{\RR}^*\ww_2$ implies $\nm(\ww_1)=\nm(\ww_2)$, hence $(\SS,\RR)$ is confluent by~\eqref{E:Quad2}. Finally, $(\SS, \RR)$ is a presentation of~$\MM$ by~\eqref{E:QuadPres}.

\ITEM2 Since~$(\SS,\RR)$ is normalising and confluent, every $\SS$-word admits exactly one $\RR$-normal form, so~\eqref{E:NormOfRewr} makes sense and implies that~$\MM$ admits the presentation~\eqref{E:NormSys4}. Next, since~$\RR$ is quadratic, $\nm$ is length-preserving and preserves generators. Moreover, the $\RR$-normal forms of $\uu\sep\ww\sep\vv$ and of $\uu\sep\nm(\ww)\sep\vv$ are equal, whence $\nm(\uu\sep\ww\sep\vv)=\nm(\uu\sep\nm(\ww)\sep\vv)$. So~$(\SS,\nm)$ is a normalisation for~$\MM$. Moreover, the definition of~$\nm$ implies that it satisfies both~\eqref{E:Quad1} and~\eqref{E:Quad2}. 

\ITEM3 The proof is straightforward.
\end{proof}

Note that the rewriting system associated to a quadratic normalisation does not always \emph{terminate}, meaning that there may exist infinite rewriting sequences $\ww_0\fl_{\RR}\ww_1\fl_{\RR}\ww_2\fl_{\RR}\pdots$, as shown in Section~\ref{S:Termination}.

\begin{exam}\label{X:AbelianRewr}
If $(\SS,\nm)$ is the quadratic normalisation for the free commutative monoid~$\NNNN^{(\SS)}$ of Example~\ref{X:Abelian}, the associated quadratic rewriting system~$(\SS, \RR)$ contains one rule $\tt\ss\fl\ss\tt$ for all $\ss, \tt$ in~$\SS$ with $\tt>\ss$. By Proposition~\ref{P:QuadNormRewr}~\ITEM1, this rewriting system is normalising and confluent.
\end{exam}

Proposition~\ref{P:QuadNormRewr}\ITEM1 can be declined to account for a neutral element and the termination properties of the corresponding rewriting systems are related.

\begin{prop}\label{P:NeutralQuadRewr}
\ITEM1 If $(\SS,\nm)$ is a normalisation mod~$\ee$ for a monoid~$\MM$, then we obtain a reduced, normalising and confluent rewriting system~$(\SS_{\ee}, \RR_\ee)$ presenting~$\MM$ by putting 
\begin{equation}
\RR_\ee = \{\ss\sep\tt \fl \pi_{\ee}(\nmr(\ss\sep\tt)) \mid \ss,\tt\in \SS_{\ee}, \ss\sep\tt\neq\nmr(\ss\sep\tt)\}.
\end{equation}

\ITEM2 If the rewriting system~$(\SS,\RR)$ of~\eqref{E:QuadRewrOfNorm} terminates, then so does~$(\SS_\ee,\RR_\ee)$.
\end{prop}

\begin{proof}
\ITEM1 Similar to Proposition~\ref{P:QuadNormRewr}\ITEM1. 

\ITEM2 If $\ww\fl_{\RR_\ee}\ww'$ holds for $\SS_\ee$-words~$\ww,\ww'$, then, by definition, there exists a position~$\ii$ satisfying $\ww'=\pi_{\ee}(\nmr_{\!\ii}(\ww))$. Thus, there exists a sequence of positions $\uu=\ii_1\sep\pdots\sep\ii_{\pp}$ satisfying $\nmr_{\!\uu}(\nmr_{\!\ii}(\ww))=\ww'\sep\ee^{\mm}$ for some~$\mm$, and where each~$\nmr_{\!\ii_{\jj}}$ acts according to a rule~$\ee\sep\ss\fl\ss\sep\ee$. Hence, each sequence $\ww_0\fl_{\RR_\ee}\ww_1\fl_{\RR_\ee}\pdots\fl_{\RR_\ee}\ww_{\ell}$ in~$\SS_{\ee}^*$ lifts to a sequence 
$$\ww_0 \fl_\RR \ww_0' \fl_\RR^* \ww_1\sep\ee^{\mm_1}
\fl_\RR \ww'_1\sep\ee^{\mm_1} \fl_\RR^* \ww_2 \sep \ee^{\mm_1 + \mm_2} \fl_\RR^* \pdots \fl_\RR^* \ww_{\ell}\sep\ee^{\mm_1+\pdots+\mm_{\ell}}$$
in~$\SS^*$. So, if $(\SS_\ee,\RR_\ee)$ does not terminate, neither does~$(\SS,\RR)$.
\end{proof}

%%%%
\subsection{The class of a quadratic normalisation}
\label{SS:Class}

By definition, if $(\SS, \nm)$ is a quadratic normalisation and $\ww$ is an $\SS$-word, then $\nm(\ww)$ is obtained by successively applying the restriction~$\nmr$ of~$\nm$ to various length-two factors. We shall now investigate the possibilities and introduce a parameter, called the class, evaluating the complexity of the procedure for length-three $\SS$-words.

For such a~$\ww$, there must exist a finite sequence~$\uu$ of positions~$1$ and~$2$, such that, with the convention of Notation~\ref{N:Local}, $\nm(\ww)$ is equal to~$\nmr_{\!\uu}(\ww)$. As $\nmr$ is idempotent, repeating~$1$ or~$2$ in the sequence~$\uu$ is useless, and it is enough to consider alternating words~$\uu$ of the form~$121...$ or $212...$, omitting the separators to make reading easier. For $\mm \ge 0$, we write $\ai{\mm}$ for the alternating word~$121...$ of length~$\mm$, and similarly for~$\aii\mm$. So, for instance, $\nmr_{\!\aii4}$ will stand for $\nmr_{\!2121}$, that is, for the composition of $\nmr_{\!2}$, $\nmr_{\!1}$, $\nmr_{\!2}$, and $\nmr_{\!1}$ with $\nmr_{\!2}$ applied first. According to the above discussion, if $(\SS, \nm)$ is a quadratic normalisation, then, for every length-three $\SS$-word~$\ww$, there exists~$\mm$ such that $\nm(\ww)$ is $\nmr_{\!\ai\mm}(\ww)$ or $\nmr_{\!\aii\mm}(\ww)$.

\begin{defi}
For~$\cc$ a natural number, we say that a quadratic normalisation~$(\SS, \nm)$ is \emph{of left-class}~$\cc$ if $\nm(\ww) = \nmr_{\!\ai\cc}(\ww)$ holds for every~$\ww$ in~$\Pow\SS3$. Symmetrically, we say that $(\SS, \nm)$ is \emph{of right-class}~$\cc$ if $\nm(\ww) = \nmr_{\!\aii\cc}(\ww)$ holds for every~$\ww$ in~$\Pow\SS3$. We say that $(\SS, \nm)$ is \emph{of class~$(\cc, \cc')$} if it is of left-class~$\cc$ and right-class~$\cc'$.
\end{defi}

\begin{exam}\label{X:Abelian3}
Let $(\SS, \nm)$ be the lexicographic normalisation for~$\NNNN^\SS$ of Example~\ref{X:Abelian}. For $\card\SS = 1$, there is only one length-three $\SS$-word, which is $\nm$-normal, so $(\SS, \nm)$ is of class~$(0,0)$. Assume now $\card\SS \ge 2$. Then, one checks that, for all~$\rr, \ss, \tt$ in~$\SS$, the words $\nmr_{\!121}(\rr\sep\ss\sep \tt)$ and $\nmr_{\!212}(\rr\sep\ss\sep \tt)$ are $\nm$-normal (and equal), so $(\SS, \nm)$ is of class~$(3,3)$. On the other hand, assuming $\tta < \ttb$, we find $\nmr_{\!12}(\ttb\ttb\tta) = \ttb\tta\ttb$ and $\nmr_{\!21}(\ttb\tta\tta) = \tta\ttb\tta$, so $(\SS, \nm)$ is neither of left-class~$2$ nor of right-class~$2$. 

To give another example, consider $\SS = \{\tta, \ttb\}$ and~$\nm$ defined by $\nm(\ww) = \tta^{\LG\ww}$ if~$\ww$ contains an even number of letters~$\ttb$, and $\nm(\ww) = \tta^{\LG\ww-1}\ttb$ otherwise. One checks that $(\SS, \nm)$ is a quadratic normalisation for the monoid $\PRESp{\tta, \ttb}{\tta\ttb = \ttb\tta, \tta^2 = \ttb^2}$. Then, a case-by-case checking on~$\Pow\SS3$ shows that $(\SS, \nm)$ is of class~$(2, 3)$, but neither of left-class~$1$ nor of right-class~$2$, as shows the worst-case example
\[
\ttb\tta\tta 
\quad (\ \buildrel {\textstyle\nmr_{\!2}\ } \over \longrightarrow \ \ttb\tta\tta\ ) 
\quad \buildrel {\textstyle \nmr_{\!1}\ } \over \longrightarrow \quad \tta\ttb\tta 
\quad \buildrel {\textstyle \nmr_{\!2}\ } \over \longrightarrow \quad \tta\tta\ttb.
\]
\end{exam}

The following observation, already implicit in the above example, will be crucial.

\begin{lemm}\label{L:Class}
Assume that~$(\SS, \nm)$ is a quadratic normalisation.

\ITEM1 If~$\ww$ is in $\Pow{\SS}{3}$, then $\nm(\ww) = \nmr_{\!\ai\cc}(\ww)$ implies $\nm(\ww) = \nmr_{\!\ai{\cc+1}}(\ww)$.

\ITEM2 If $(\SS, \nm)$ is of left-class~$\cc$, then it is of left-class~$\cc'$ for every~$\cc'$ with $\cc' \ge \cc$, and of right-class~$\cc''$ for every~$\cc''$ with $\cc'' \ge \cc+1$.
\end{lemm}

\begin{proof}
\ITEM1 Assume $\nm(\ww) = \nmr_{\!\ai\cc}(\ww)$. By~\eqref{E:Quad1}, $\nm(\ww)$ is invariant both under~$\nmr_{\!1}$ and $\nmr_{\!2}$, since it is $\nm$-normal. Hence we have $\nmr_{\!\ai{\cc+1}}(\ww) = \nmr_{\!\ai\cc}(\ww)$.

\ITEM2 Assume that $(\SS, \nm)$ is of left-class~$\cc$. Then~\ITEM1 implies $\nm(\ww) = \nmr_{\!\ai{\cc+1}}(\ww)$ for every~$\ww$ in~$\Pow\SS3$, so $(\SS, \nm)$ is of left-class~$\cc+1$ as well and, from there, it is of left-class~$\cc'$ for every~$\cc' \ge \cc$. For~$\ww$ in~$\Pow\SS3$, the assumption and~\eqref{E:NormSys3} give $\nm(\ww) = \nmr_{\!\ai\cc}(\nmr_{\!2}(\ww)) = \nmr_{\!\aii{\cc+1}}(\ww)$. Hence $(\SS, \nm)$ is of right-class~$\cc+1$ and, from there, of right-class~$\cc''$ for every~$\cc''$ with $\cc'' \ge \cc+1$. 
\end{proof}

Define the \emph{minimal left-class} of a quadratic normalisation~$(\SS, \nm)$ to be the smallest integer~$\cc$ such that $(\SS, \nm)$ is of left-class~$\cc$, if such an integer exists, and~$\infty$ otherwise. We introduce the symmetric notion of \emph{minimal right-class}, and define the \emph{minimal class} to be the pair made of the minimal left-class and the minimal right-class.

\begin{lemm}\label{L:Constraints}
The minimal class of a quadratic normalisation~$(\SS, \nm)$ is either of the form~$(\cc, \cc')$ with $\vert\cc' - \cc\vert \le 1$, or~$(\infty, \infty)$. If~$\SS$ is finite, the value~$(\infty, \infty)$ is excluded.
\end{lemm}

\begin{proof}
If the minimal left-class of~$(\SS, \nm)$ is a finite number~$\cc$, then Lemma~\ref{L:Class} implies that $(\SS, \nm)$ is of right-class~$\cc+1$; hence the minimal right-class~$\cc'$ satisfies $\cc' \le \cc+1$ and, for symmetric reasons, we have $\cc \le \cc' +1$, whence $\vert\cc' - \cc\vert \le 1$.

The assumption that $(\SS, \nm)$ is quadratic implies, for every~$\ww$ in~$\Pow\SS3$, the existence of a smallest finite number~$\cc_{\ww}$ satisfying $\nm(\ww) = \nmr_{\!\ai{\cc_{\ww}}}(\ww)$. If~$\SS$ is finite, the supremum of all numbers~$\cc_\ww$ for~$\ww$ in~$\Pow\SS3$ is finite, and Lemma~\ref{L:Class}~\ITEM1 implies that~$\cc$ is the minimal left-class of~$(\SS, \nm)$.
\end{proof}

The class of a normalisation~$(\SS, \nm)$ can be characterised in terms of algebraic relations exclusively involving the map~$\nmr$.

\begin{prop}\label{P:Class}
A quadratic normalisation~$(\SS, \nm)$ is of left-class~$\cc$ if, and only if, $\nmr$ satisfies
\begin{equation}\label{E:Class1}
\nmr_{\!\ai\cc} = \nmr_{\!\ai{\cc+1}} = \nmr_{\!\aii{\cc+1}},
\end{equation}
and of class~$(\cc, \cc)$ if, and only if, $\nmr$ satisfies
\begin{equation}\label{E:Class2}
\nmr_{\!\ai\cc} = \nmr_{\!\aii\cc}.
\end{equation}
\end{prop}

\begin{proof}
Assume that $(\SS, \nm)$ is of left-class~$\cc$. For every~$\ww$ in~$\Pow\SS3$, Lemma~\ref{L:Class} gives $\nm(\ww) = \nmr_{\!\ai\cc}(\ww) = \nmr_{\!\ai{\cc+1}}(\ww) = \nmr_{\!\aii{\cc+1}}(\ww)$, whence~\eqref{E:Class1}. Conversely, assume~\eqref{E:Class1} and let~$\ww$ belong to~$\Pow\SS3$. If~$\cc$ is odd, we obtain
$$\nmr_{\!1}(\nmr_{\!\ai\cc}(\ww)) = \nmr_{\!1}(\nmr_{\!1}(\nmr_{\!\ai{\cc-1}}(\ww))) = \nmr_{\!\ai\cc}(\ww),$$
since~$\nmr$ is idempotent, and, by~\eqref{E:Class1}, 
$$\nmr_{\!2}(\nmr_{\!\ai\cc}(\ww)) = \nmr_{\!\ai{\cc+1}}(\ww) = \nmr_{\!\ai\cc}(\ww).$$
If~$\cc$ is even, a symmetric argument gives the same values. So, in all cases, the $\SS$-word $\nmr_{\!\ai\cc}(\ww)$ is invariant both under~$\nmr_{\!1}$ and~$\nmr_{\!2}$, hence it is $\nm$-normal. As this holds for every~$\ww$ in~$\Pow\SS3$, we conclude that $(\SS, \nm)$ is of left-class~$\cc$.

Assume that~$(\SS, \nm)$ is of class~$(\cc, \cc)$. By~\ITEM1, we have $\nmr_{\!\ai\cc} = \nmr_{\!\ai{\cc+1}} = \nmr_{\!\aii{\cc+1}}$ and, by the symmetric counterpart of~\ITEM1, we have $\nmr_{\!\aii\cc} = \nmr_{\!\aii{\cc+1}} = \nmr_{\!\ai{\cc+1}}$, whence~\eqref{E:Class2} by merging the values. Conversely, assume~\eqref{E:Class2} and let~$\ww$ belong to~$\Pow\SS3$. Applying~\eqref{E:Class2} to~$\nmr_{\!1}(\ww)$ gives $\nmr_{\!\ai\cc}(\nmr_{\!1}(\ww)) = \nmr_{\!\aii\cc}(\nmr_{\!1}(\ww))$, reducing to $\nmr_{\!\ai\cc}(\ww) = \nmr_{\!\ai{\cc+1}}(\ww)$, since~$\nmr_{\!1}$ is idempotent. Similarly, applying~\eqref{E:Class2} to~$\nmr_{\!2}(\ww)$ leads to $\nmr_{\!\aii{\cc+1}}(\ww) = \nmr_{\!\aii\cc}(\ww)$. Merging the results and applying~\eqref{E:Class2} to~$\ww$, we deduce
$$\nmr_{\!\ai{\cc+1}}(\ww) = \nmr_{\!\ai\cc}(\ww) = \nmr_{\!\aii\cc}(\ww) = \nmr_{\!\aii{\cc+1}}(\ww).$$
As this holds for every~$\ww$ in~$\Pow\SS3$, \eqref{E:Class1} is satisfied, so, by~\ITEM1, $(\SS, \nm)$ is of left-class~$\cc$. A symmetric argument implies that $(\SS, \nm)$ is of right-class~$\cc$.
\end{proof}

The following example shows that the minimal left-class of a quadratic normalisation can be an arbitrarily high integer. 

\begin{exam}\label{X:High3}
For $\nn \ge 2$, let $\SS_\nn = \{\tta, \ttb_1 \wdots \ttb_\nn\}$ and~$\RR_{\nn}$ consist of the rules $\tta\ttb_\ii \to \tta\ttb_{\ii+1}$ for~$\ii < \nn$ odd and $\ttb_\ii\tta \to \ttb_{\ii+1}\tta$ for~$\ii < \nn$ even. Then the rewriting system $(\SS_{\nn},\RR_{\nn})$ is convergent: termination is given by comparison of the number of~$\ttb_1$, then of~$\ttb_2$, then of~$\ttb_3$, etc.; confluence is obtained by observing that, for every minimal overlapping application of rules $\ttb_{\ii}\tta\ttb_{\jj}\fl_{\RR_{\nn}}\ttb_{\ii+1}\tta\ttb_{\jj}$ and $\ttb_{\ii}\tta\ttb_{\jj}\fl_{\RR_{\nn}}\ttb_{\ii}\tta\ttb_{\jj+1}$, we have $\ttb_{\ii+1}\tta\ttb_{\jj}\fl_{\RR_{\nn}}\ttb_{\ii+1}\tta\ttb_{\jj+1}$ and $\ttb_{\ii}\tta\ttb_{\jj}\fl_{\RR_{\nn}}\ttb_{\ii+1}\tta\ttb_{\jj+1}$. 
Let $(\SS_n,\nm_\nn)$ be the associated quadratic normalisation as defined in Proposition~\ref{P:QuadNormRewr}. For $\nn \ge 3$, the minimal class of~$(\SS_\nn, \nm_\nn)$ is $(\nn-1, \nn)$: length-three words that do not begin and finish with~$\tta$ are $\RR_n$-normal or become $\RR_n$-normal in one step, and the reduction of~$\tta\ttb_1\tta$ looks like
$$\tta\ttb_1\tta 
\quad (\ \buildrel {\textstyle \nmr_{\!2}\ } \over \longrightarrow \ \tta\ttb_1\tta\ ) 
\quad \buildrel {\textstyle \nmr_{\!1}\ } \over \longrightarrow \quad \tta\ttb_2\tta 
\quad \buildrel {\textstyle \nmr_{\!2}\ } \over \longrightarrow \quad \tta\ttb_3\tta 
\quad \buildrel {\textstyle \nmr_{\!1}\ } \over \longrightarrow \quad \pdots 
\quad \buildrel {\textstyle \nmr_-\ } \over \longrightarrow \quad \tta\ttb_\nn\tta,$$
implying that the minimal left-class is~$\nn-1$, and the minimal right-class is~$\nn$.
\end{exam}

The next example shows that the minimal left-class can be~$\infty$. (Putting~$\nn=\infty$ in Example~\ref{X:High3} provides a non-normalising system: $\tta\ttb_1\tta$ has no normal form.)

\begin{exam}
For $\nn \ge 2$, let $\SS_{\nn} = \{\tta_0 \wdots \tta_\nn\}$ and $\RR_{\nn}$ consist of the rules 
$$\tta_\ii\tta_\jj \to \tta_{\lfloor \frac{\scriptstyle\ii+\jj}{\scriptstyle2}\rfloor}\tta_{\lceil \frac{\scriptstyle\ii+\jj}{\scriptstyle2}\rceil} \text{\quad for~$\ii >\jj$}.$$
The rewriting system $(\SS_{\nn},\RR_{\nn})$ is convergent. Let $(\SS_{\nn},\nm_{\nn})$ be the associated quadratic normalisation. As in Example~\ref{X:Abelian}, the $\nm_\nn$-normal words are the lexicographically non-decreasing ones with respect to $\tta_1 < \pdots < \tta_\nn$. Then the minimal class of~$(\SS_\nn, \nm_\nn)$ is $(3{+}\lfloor\log_2\nn\rfloor, 3{+}\lfloor\log_2\nn\rfloor)$. Indeed, for $2^\pp \le \nn < 2^{\pp+1}$, the worst case for the left-class is attained by $\tta_{2^\pp}\tta_{2^\pp}\tta_0$: putting $\ii = \lfloor2^{\pp+1}/3\rfloor$, the latter $\SS_{\nn}$-word reduces to $\tta_\ii\tta_\ii\tta_{\ii+1}$ (even $\pp$) or $\tta_\ii\tta_{\ii+1}\tta_{\ii+1}$ (odd $\pp$) in $\pp + 3$ steps. Moreover, if we define $\SS_\infty$ to be the infinite set $\{\tta_0, \tta_1, ...\}$ and~$\nm_\infty$ associated as above, $(\SS_\infty, \nm_\infty)$ is a quadratic normalisation with minimal class $(\infty, \infty)$ as, for every~$\pp$, the reduction of $\tta_{2^\pp}\tta_{2^\pp}\tta_0$ requires $\pp+3$ steps. 
\end{exam}

%%%%
\subsection{The $\pp$-class}\label{SS:pClass}

We now consider the normalisation of length-$\pp$ words for $\pp \ge 4$. If $(\SS, \nm)$ is a quadratic normalisation, then, by definition, one can transform a length-$\pp$ word~$\ww$ into~$\nm(\ww)$ by applying a finite number of elementary maps~$\nmr_{\!\ii}$ with $1 \le \ii < \pp$. Contrary to the case $\pp = 3$, there may exist many ways of composing these maps for $\pp \ge 4$: for instance, one can consider the left strategy consisting in always normalising the leftmost unreduced length-two factor of the current word, but this choice is arbitrary. Writing $\nmR\pp$ for the restriction of~$\nm$ to $\SS$-words of length~$\pp$, more canonical decompositions arise when one expresses~$\nmR4$ in terms of~$\nmR3$ and, more generally, $\nmR\pp$ in terms of~$\nmR{\pp-1}$. Then the situation is similar to~$3$ {\it vs.}~$2$, and a natural notion of $\pp$-class appears. 

\begin{defi}
For $\pp \ge 3$, we say that a quadratic normalisation~$(\SS, \nm)$ is \emph{of left-$\pp$-class}~$\cc$ if, for every~$\ww$ in~$\Pow\SS\pp$, we have $\nm(\ww) = \nmR{\pp-1}_{\ai\cc}$. Symmetrically, we say that $(\SS, \nm)$ is \emph{of right-$\pp$-class}~$\cc$ if, for every~$\ww$ in~$\Pow\SS\pp$, we have $\nm(\ww) = \nmR{\pp-1}_{\aii\cc}$. We say that $(\SS, \nm)$ is \emph{of $\pp$-class~$(\cc, \cc')$} if it is of left-$\pp$-class~$\cc$ and right-$\pp$-class~$\cc'$.
\end{defi}

Thus the left-class of Subsection~\ref{SS:Class} is the left-$3$-class, and similarly for the right-class and the class.

\begin{exam}
Consider the lexicographic normalisation $(\SS, \nm)$ of Example~\ref{X:Abelian3}. We saw that, for $\card\SS \ge 2$, the minimal ($3$)-class is $(3, 3)$. An easy induction shows that, for every~$\pp \ge 4$ and for every~$\ww$ in~$\Pow{\SS}\pp$, the words $\nmR{\pp-1}_{212}(\ww)$ and $\nmR{\pp-1}_{121}(\ww)$ are lexicographically nondecreasing, hence $\nm$-normal. Thus $(\SS, \nm)$ is of $\pp$-class~$(3, 3)$. Then, assuming $\card\SS \ge 2$ and $\tta < \ttb$, we find $\nmR{\pp-1}_{12}(\ttb^{\pp-1}\tta) = \ttb\tta\ttb^{\pp-2}$ and $\nmR{\pp-1}_{21}(\ttb\tta^{\pp-1}) = \tta^{\pp-2}\ttb\tta$, which are not $\nm$-normal. So $(\SS, \nm)$ is neither of left-$\pp$-class~$2$ nor of right-$\pp$-class~$2$. 
\end{exam}

The $\nm$-normality of $\SS$-words can be characterised in terms of $\nm$-normality of their length-$\pp$ factors, with a straightforward proof:

\begin{lemm}
If $(\SS, \nm)$ is a quadratic normalisation, then, for $\pp \ge2$, an $\SS$-word~$\ww$ with $\LG\ww \ge \pp$ is $\nm$-normal if, and only if, every length-$\pp$ factor of~$\ww$ is.
\end{lemm}

All properties of the $3$-class extend to the $\pp$-class for $\pp \ge 3$. In particular, when it is not~$(\infty, \infty)$, the minimal $\pp$-class must be a pair of the form~$(\cc, \cc')$ with~$\vert \cc-\cc' \vert \le 1$, and we have the following counterpart of Proposition~\ref{P:Class}, with a similar proof: 

\begin{prop}\label{P:PClass}
A quadratic normalisation~$(\SS, \nm)$ is of left-$\pp$-class~$\cc$ if, and only if, the map $\nmR{\pp-1}$ satisfies
$\nmR{\pp-1}_{\ai\cc} = \nmR{\pp-1}_{\ai{\cc+1}} = \nmR{\pp-1}_{\aii{\cc+1}}$,
and of $\pp$-class~$(\cc, \cc)$ if, and only if, the map $\nmR{\pp-1}$ satisfies
$\nmR{\pp-1}_{\ai\cc} = \nmR{\pp-1}_{\aii\cc}.$
\end{prop}

The following examples show that the behaviour of the $4$-class is independent from that of the $3$-class: the $4$-class may be larger, equal, or smaller.

\begin{exam}
The normalisation~$(\SS_\nn, \nm_\nn)$ of Example~\ref{X:High3} has minimal $3$-class is $(\nn, \nn)$. However, for $\pp \ge 4$, its minimal $\pp$-class is~$(2, 2)$.
\end{exam}

\begin{exam}
Let $\SS_\nn = \{\tta, \ttb_1 \wdots \ttb_\nn\}$ and $\nm_\nn$ be given by the rules $\tta\ttb_\ii {\to} \tta\ttb_{\ii+1}$ for~$\ii$ odd and $\ttb_\ii\tta {\to} \ttb_{\ii+1}\tta$ for~$\ii$ even (as in Example~\ref{X:High3}), completed with $\ttb_{\ii+1}\ttb_\ii \fl \ttb_{\ii+1} \ttb_{\ii+1}$ for~$\ii$ odd and $\ttb_\ii\ttb_{\ii+1} {\to} \ttb_{\ii+1}\ttb_{\ii+1}$ for~$\ii$ even. For $\pp\ge 3$, the minimal $\pp$-class of~$(\SS_{\nn},\nm_{\nn})$ is~$(\nn-1, \nn)$, with the worst case realised for $\tta\ttb_1^{\pp-2}\tta$.
\end{exam}

\begin{exam}\label{X:Large4Class}
Let $\SS_\nn = \{\tta, \ttb_1 \wdots \ttb_\nn, \ttc_1 \wdots \ttc_\nn\}$ and $\nm_\nn$ be given by the rules $\tta\ttb_\ii {\to} \tta\ttb_{\ii+1}$ and $\ttb_{\ii+1}\ttc_\ii {\to} \ttb_{\ii+1} \ttc_{\ii+1}$ for~$\ii$ odd, and $\ttc_\ii\tta {\to} \ttc_{\ii+1}\tta$ and $\ttb_\ii\ttc_{\ii+1} {\to} \ttb_{\ii+1}\ttc_{\ii+1}$ for~$\ii$ even. Here it turns out that the minimal $3$-class is~$(5, 5)$, whereas the minimal $4$-class is~$(\nn-1, \nn)$. For instance, for $\nn = 10$, the worst cases are realised by $\nmr_{\!12121}(\ttb_4\ttc_2\tta) = \ttb_5\ttc_5\tta$, and $\nmR3_{2121212121}(\tta\ttb_1\ttc_1\tta) = \tta\ttb_{10}\ttc_{10}\tta$. One also observes that the minimal $\pp$-class for $\pp \ge 5$ is~$(2,2)$ for every $\nn \ge 5$. 
\end{exam}

%%%%
\section{Quadratic normalisations of class $(4, 3)$}\label{S:Class3}

Example~\ref{X:Large4Class} shows that having left-class or right-class~$\cc$ does not say much about normalisation of words of length four and higher: an upper bound on the complexity of normalisation for length-three words implies no upper bound on the complexity of normalisation for longer words. We observe below that such a phenomenon is impossible when the class is small, namely when the class is~$(3, 4)$ or~$(4, 3)$. Our proof is based on an argument borrowed from~\cite{Garside}, involving a diagrammatic approach called the domino rule. The results for classes~$(3, 4)$ and $(4, 3)$ are entirely similar; the latter is chosen here in view of the connection with Garside normalisation in Section~\ref{S:Garside}.

The section comprises three parts. First, the domino rule is introduced in Subsection~\ref{SS:Domino}. Next, we establish a general formula for normalisation of long words when some domino rule is valid in Subsection~\ref{SS:Long}. Finally, we show in Subsection~\ref{SS:Axiomatisation} how standard braid arguments can be used to provide a complete axiomatisation of class~$(4, 3)$ normalisation.

%%%%
\subsection{The domino rule}\label{SS:Domino}

By Proposition~\ref{P:Class}, if a quadratic normalisation $(\SS, \nm)$ has class~$(4, 3)$, the map~$\nmr$ satisfies~\eqref{E:Class1}, which, in the current case, is 
\begin{equation}\label{E:Class3}
\nmr_{\!212} = \nmr_{\!2121} = \nmr_{\!1212}.
\end{equation}
We shall now translate these conditions into a diagrammatic rule.

\begin{defi}\label{D:Domino}
Assume that $\SS$ is a set and $\phi$ is a map from~$\Pow\SS2$ to itself. We say that the \emph{domino rule is valid for~$\phi$} if, if, for all $\ss_1, \ss_2, \ss'_1, \ss'_2, \tt_0, \tt_1, \tt_2$ in~$\SS$ satisfying $\ss'_1 \sep \tt_1 = \phi(\tt_0, \ss_1)$ and $\ss'_2 \sep \tt_2 = \phi(\tt_1 \sep \ss_2)$, the assumption that $\ss_1\vert\ss_2$ is $\phi$-invariant implies that $\ss'_1\vert\ss'_2$ is $\phi$-invariant as well.
\end{defi}

The domino rule of Definition~\ref{D:Domino} becomes more understandable when illustrated in a diagram. To this end, we associate with every element~$\ss$ of the considered set~$\SS$ an $\ss$-labeled arrow, and use concatenation of arrows for the concatenation of elements (note that this amounts to viewing $\SS^*$ as a category). 

\rightskip25mm
Let us indicate that a word~$\ss\sep\tt$ of~$\Pow\SS2$ is $\ft$-invariant---hence~$\nm$-normal when~$\ft$ is the map~$\nmr$ associated with a normalisation~$(\SS, \nm)$---with a small arc, as in \begin{picture}(30,4)
\psarc[style=thin](14.5,0){3}{0}{180}
\pcline{->}(1,0)(14,0)
\taput{$\ss$}
\pcline{->}(16,0)(29,0)
\taput{$\tt$}
\end{picture}.
Then, in the situation when $\ss'\sep\tt' = \ft(\ss\sep\tt)$ holds, we draw a square diagram as on the right.\hfill%
\begin{picture}(0,0)(-6,-5)
\pcline{->}(0,9)(0,1)\tlput{$\ss$}
\pcline{->}(1,0)(14,0)\tbput{$\tt$}
\pcline{->}(1,10)(14,10)\taput{$\ss'$}
\pcline{->}(15,9)(15,1)\trput{$\tt'$}
\psarc[style=thin](15,10){3}{180}{270}
\end{picture}

\rightskip42mm
With such conventions, the domino rule for~$\ft$ corresponds to the diagram on the right: whenever the two squares are commutative and the three pairs of edges connected with small arcs are $\ft$-invariant, then so is the fourth pair indicated with a dotted arc.\hfill%
\begin{picture}(0,0)(-8,-2)
\psarc[style=thin](15,0){3}{180}{360}
\psarc[style=thin](15,10){3.5}{180}{270}
\psarc[style=thin](30,10){3.5}{180}{270}
\psarc[style=thinexist](15,10){3}{0}{180}
\pcline{->}(1,0)(14,0)
\tbput{$\ss_1$}
\pcline{->}(16,0)(29,0)
\tbput{$\ss_2$}
\pcline{->}(1,10)(14,10)
\taput{$\ss'_1$}
\pcline{->}(16,10)(29,10)
\taput{$\ss'_2$}
\pcline{->}(0,9)(0,1)
\tlput{$\tt_0$}
\pcline{->}(15,9)(15,1)
\trput{$\tt_1$}
\pcline{->}(30,9)(30,1)
\trput{$\tt_2$}
\end{picture}

\rightskip0mm

\begin{lemm}\label{L:Domino}
A quadratic normalisation~$(\SS, \nm)$ is of class~$(4, 3)$ if, and only if, the domino rule is valid for~$\nmr$.
\end{lemm}

\begin{proof}
Assume that~$(\SS, \nm)$ is of right-class~$3$, and let $\ss_1 \wdots \tt_2$ be elements of~$\SS$ satisfying the assumptions of the domino rule. By definition of the right-class, we have $\nm(\tt_0 \sep \ss_1 \sep \ss_2) = \nmr_{\!212}(\tt_0 \sep \ss_1 \sep \ss_2)$. As, by assumption, $\ss_1\sep\ss_2$ is $\nm$-normal, we obtain
$\nm(\tt_0 \sep \ss_1 \sep \ss_2) = \nmr_{\!12}(\tt_0 \sep \ss_1 \sep \ss_2) = \nmr_{\!2}(\ss'_1\sep\tt_1\sep\ss_2) = \ss'_1\sep\ss'_2\sep\tt_2$. So $\ss'_1\sep\ss'_2\sep\tt_2$ is $\nm$-normal, hence so is $\ss'_1\sep\ss'_2$, and the domino rule is valid for~$\nmr$. 

Conversely, assume that the domino rule is valid for~$\nmr$. Let $\tt_0 \sep \rr_1 \sep \rr_2$ be an arbitrary word in~$\Pow\SS3$. Put $\ss_1 \sep \ss_2 = \nmr(\rr_1 \sep \rr_2)$, $\ss'_1 \sep \tt_1 = \nmr(\tt_0 \sep \ss_1)$, and $\ss'_2 \sep \tt_2 = \nmr(\tt_1 \sep \ss_2)$. Then $\ss'_2 \sep \tt_2$ is $\nm$-normal by construction, and $\ss'_1 \sep \ss'_2$ is $\nm$-normal by the domino rule, so $\ss'_1 \sep \ss'_2 \sep \tt_2$ is $\nm$-normal. Hence we have $\nm(\ww) = \nmr_{\!212}(\ww)$ for every~$\ww$ in~$\Pow\SS3$, and $(\SS, \nm)$ is of right-class~$3$. 
\end{proof}

%%%%
\subsection{Normalising long words}\label{SS:Long}

We shall now show that, is~$(\SS, \nm)$ is a normalisation map of class~$(4, 3)$, there exists a simple formula for the normalisation of arbitrarily long words.

\begin{nota}\label{N:Iteration}
Starting from $\delta_1 = \varepsilon$ (the empty sequence) and using $\sh$ for the shift mapping that increases every entry by~$1$, we inductively define a finite sequence of positive integers~$\delta_\pp$ by
\begin{equation}\label{E:delta}
\delta_\pp = \sh(\delta_{\pp-1}) \sep 1 \sep 2 \sep \pdots \sep \pp-2 \sep \pp-1 \text{\qquad for $\pp \ge 1$.}
\end{equation}
\end{nota}

Thus we find, omitting the separation symbol,
$$\delta_2 = 1, \quad \delta_3 = 212, \quad \delta_4 = 323123, \quad \delta_5 = 4342341234, \quad \mbox{etc.}$$

\begin{prop}\label{P:NormClass3}
Assume that~$(\SS, \nm)$ is a quadratic normalisation of class~$(4, 3)$. Then, for every $\pp \ge 1$ and every length-$\pp$ word~$\ww$, we have 
\begin{equation}
\label{E:NormClass3}
\nm(\ww) = \nmr_{\!\delta_\pp}(\ww).
\end{equation} 
\end{prop}

So there is a universal recipe, prescribed by the sequence of positions~$\delta_\pp$, for normalising every word of length~$\pp$. We begin with a preparatory result.

\begin{lemm}\label{L:LeftMult}
If $(\SS, \nm)$ is a quadratic normalisation and the domino rule is valid for~$\nmr$, then, for every~$\tt$ in~$\SS$ and every $\nm$-normal $\SS$-word $\ss_1 \sep \pdots \sep \ss_\qq$, we have
\begin{equation}\label{E:MultOne}
\nm(\tt \sep \ss_1 \sep \pdots \sep \ss_\qq) = \nmr_{\!12 \pdots (\qq-1)\qq}(\tt \sep \ss_1 \sep \pdots \sep \ss_\qq).
\end{equation}
\end{lemm}

\begin{proof}
For $\qq = 1$, \eqref{E:MultOne} reduces to $\nm(\tt \sep \ss_1) = \nmr(\tt \sep \ss_1)$. Assume $\qq \ge 2$. Put $\tt_0 = \tt$ and inductively define~$\ss'_\ii$ and~$\tt_\ii$ by $\ss'_\ii \sep \tt_\ii = \nmr(\tt_{\ii-1} \sep \ss_\ii)$ for $\ii = 1 \wdots \qq$ (see Figure~\ref{F:LeftMult}). Then, by definition, we have 
\[
\ss'_1 \sep \pdots \sep \ss'_\qq \sep \tt_\qq = \nmr_{\!12 \pdots (\qq-1)\qq}(\tt \sep \ss_1 \sep \pdots \sep \ss_\qq),
\]
so, in order to establish~\eqref{E:MultOne}, it suffices to show that the word $\ss'_1 \sep \pdots \sep \ss'_\qq \sep \tt_\qq$ is $\nm$-normal. Now, for $\ii = 1 \wdots \qq-1$, the assumption that $\ss_\ii \sep \ss_{\ii+1}$ is $\nm$-normal and the validity of the domino rule imply that $\ss'_\ii \sep \ss'_{\ii+1}$ is $\nm$-normal as well. Finally, $\ss'_\qq \sep \tt_\qq$ is $\nm$-normal by construction. Hence, every length-two factor of $\ss'_1 \sep \pdots \sep \ss'_\qq \sep \tt_\qq$ is $\nm$-normal and, therefore, the latter is $\nm$-normal.
\end{proof}

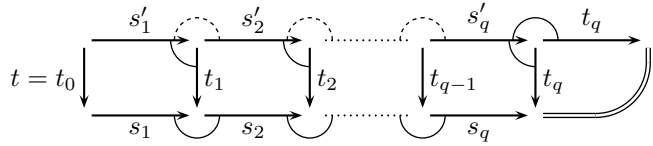
\begin{figure}[htb]
\begin{picture}(75,13)(0,-2)
\pcline{->}(1,0)(14,0)\tbput{$\ss_1$}
\pcline{->}(16,0)(29,0)\tbput{$\ss_2$}
\pcline[style=etc](32,0)(43,0)
\pcline{->}(46,0)(59,0)\tbput{$\ss_\qq$}
\pcline{->}(1,10)(14,10)\taput{$\ss'_1$}
\pcline{->}(16,10)(29,10)\taput{$\ss'_2$}
\pcline[style=etc](32,10)(43,10)
\pcline{->}(46,10)(59,10)\taput{$\ss'_\qq$}
\pcline{->}(61,10)(74,10)\taput{$\tt_\qq$}
\pcline{->}(0,9)(0,1)\tlput{$\tt = \tt_0$}
\pcline{->}(15,9)(15,1)\trput{$\tt_1$}
\pcline{->}(30,9)(30,1)\trput{$\tt_2$}
\pcline{->}(45,9)(45,1)\trput{$\tt_{\qq-1}$}
\pcline{->}(60,9)(60,1)\trput{$\tt_\qq$}
\pcline[style=double](61,0)(68,0)
\pcline[style=double](75,9)(75,7)
\psarc[style=double](68,7){7}{270}{360}
\psarc[style=thin](15,10){3.5}{180}{270}
\psarc[style=thin](30,10){3.5}{180}{270}
\psarc[style=thin](60,10){3.5}{180}{270}
\psarc[style=thin](15,0){3}{180}{360}
\psarc[style=thin](30,0){3}{180}{360}
\psarc[style=thin](45,0){3}{180}{360}
\psarc[style=thinexist](15,10){3}{0}{180}
\psarc[style=thinexist](30,10){3}{0}{180}
\psarc[style=thinexist](45,10){3}{0}{180}
\psarc[style=thin](60,10){3}{0}{180}
\end{picture}
\caption[]{\sf\smaller Left-multiplying a normal word by an element of~$\SS$: the domino rule guarantees that the upper row is normal whenever the lower row is. The diagram shows, in particular, that the monoid presented by~$(\SS, \nm)$ satisfies the $2$-Fellow Traveler Property with respect to left-multiplication~\cite{HoTh}.}
\label{F:LeftMult}
\end{figure}

\begin{proof}[Proof of Proposition~\ref{P:NormClass3}]
By Lemma~\ref{L:Domino}, the domino rule is valid for~$\nm$, hence Lemma~\ref{L:LeftMult} applies. We prove \eqref{E:NormClass3} using induction on~$\pp$. For $\pp \le 2$ , the result is immediate. Assume $\pp \ge 3$ and let $\ww = \ss_1 \vdots \ss_\pp$ belong to~$\Pow\SS\pp$. Put 
\begin{equation}\label{E:NormClass31}
\ss'_2 \vdots \ss'_\pp := \nmr_{\!\delta_{\pp-1}}(\ss_2 \vdots \ss_\pp)
\end{equation}
(see Figure~\ref{F:NormClass3}). By induction hypothesis, the word $\ss'_2 \vdots \ss'_\pp$ is $\nm$-normal. On the other hand, by definition of position shifting, \eqref{E:NormClass31} implies
\begin{equation}\label{E:NormClass32}
\ss_1 \sep \ss'_2 \vdots \ss'_\pp = \nmr_{\!\sh(\delta_{\pp-1})}(\ss_1 \sep \ss_2 \vdots \ss_\pp).\end{equation}
Then, as $\ss'_2 \vdots \ss'_{\pp-1}$ is $\nm$-normal and $\ss_1$ belongs to~$\SS$, Lemma~\ref{L:LeftMult} implies
\begin{equation*}\label{E:NormClass33}
\nm(\ss_1 \sep \ss'_2 \vdots \ss'_\pp) = \nmr_{\!12 \pdots (\pp-1)}(\ss_1 \sep \ss'_2 \vdots \ss'_\pp),
\end{equation*}
whence, owing to the inductive definition of~$\delta_\pp$,
\begin{equation}\label{E:NormClass34}
\nm(\ss_1 \sep \ss'_2 \vdots \ss'_\pp) = \nmr_{\!12 \pdots (\pp-1)} \comp \nmr_{\!\sh(\delta_{\pp-1})}(\ss_1 \sep \ss_2 \vdots \ss_\pp) = \nmr_{\!\delta_\pp}(\ss_1 \sep \ss_2 \vdots \ss_\pp).
\end{equation}
Owing to~\eqref{E:NormSys3}, \eqref{E:NormClass32} implies that $\nm(\ss_1 \sep \ss_2 \vdots \ss_\pp)$ and $\nm(\ss_1 \sep \ss'_2 \vdots \ss'_\pp)$ are equal. Merging with~\eqref{E:NormClass34}, we deduce $\nm(\ss_1 \sep \ss_2 \vdots \ss_\pp)=\nmr_{\!\delta_\pp}(\ss_1 \sep \ss_2 \vdots \ss_\pp)$.
\end{proof}

The inductive normalisation process described in Proposition~\ref{P:NormClass3} and Figure~\ref{F:NormClass3} amounts to using the map~$\nmr$ to construct a triangular grid as shown in Figure~\ref{F:Grid}.
\begin{figure}[htb]
\begin{picture}(95,24)(0,-1)
\pcline{->}(1,0)(14,0)\tbput{$\ss_1$}
\pcline{->}(16,0)(29,0)\tbput{$\ss_2$}
\pcline[style=etc](32,0)(43,0)
\pcline{->}(46,0)(59,0)\tbput{$\ss_\pp$}
\pcline[style=double](0,9)(0,1)
\pcline[style=double](15,9)(15,1)
\pcline[style=double](60,9)(60,1)
\pcline{->}(1,10)(14,10)\tbput{$\ss_1$}
\pcline{->}(16,10)(29,10)\tbput{$\ss'_2$}
\pcline[style=etc](32,10)(43,10)
\pcline{->}(46,10)(59,10)\tbput{$\ss'_\pp$}
\psarc[style=thin](30,10){3}{180}{360}
\psarc[style=thin](45,10){3}{180}{360}
\pcline[style=double](0,13)(0,11)
\pcline[style=double](7,20)(14,20)
\psarc[style=double](7,13){7}{90}{180}
\pcline{->}(15,19)(15,11)\tlput{$\ss_1$}
\pcline{->}(30,19)(30,11)
\pcline{->}(45,19)(45,11)
\pcline{->}(60,19)(60,11)
\pcline[style=double](75,19)(75,17)
\pcline[style=double](61,10)(68,10)
\psarc[style=double](68,17){7}{270}{360}
\pcline{->}(16,20)(29,20)\taput{$\ss''_1$}
\pcline[style=etc](32,20)(43,20)
\pcline{->}(46,20)(59,20)\taput{$\ss''_{\pp-1}$}
\pcline{->}(61,20)(74,20)\taput{$\ss''_\pp$}
\psarc[style=thin](30,20){3.5}{180}{270}
\psarc[style=thin](60,20){3.5}{180}{270}
\psarc[style=thin](30,20){3}{0}{180}
\psarc[style=thin](45,20){3}{0}{180}
\psarc[style=thin](60,20){3}{0}{180}
\psline[style=thin](77,0)(79,0)(79,9.5)(77,9.5)
\put(81,6.5){induction}
\put(80,2.5){hypothesis}
\psline[style=thin](77,10.5)(79,10.5)(79,20)(77,20)
\put(80,14.5){Lemma~\ref{L:LeftMult}}
\end{picture}
\caption[]{\sf\smaller Inductive normalisation process of a length-$\pp$ word $\ss_1 \sep \pdots \sep \ss_\pp$ based on the domino rule: first normalise $\ss_2\vdots \ss_\pp$ into $\ss'_2 \vdots \ss'_\pp$, and then normalise $\ss_1 \vert \ss'_2 \vdots \ss'_\pp$ into $\ss''_1 \vdots \ss''_\pp$, which is $\nm$-normal by Lemma~\ref{L:LeftMult}.}
\label{F:NormClass3}
\end{figure}
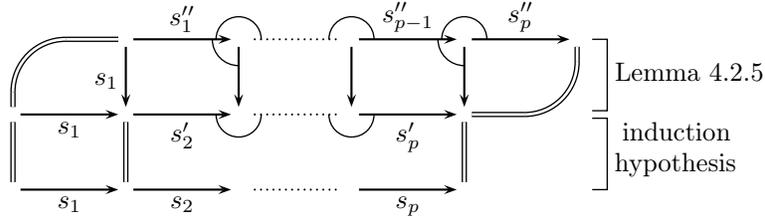

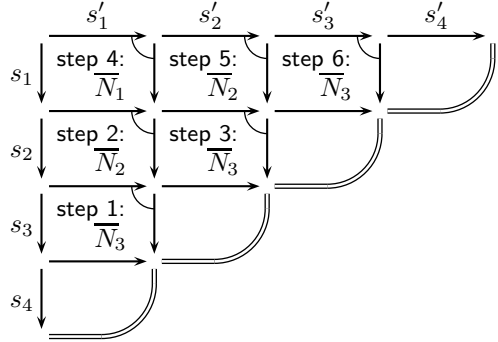
\begin{figure}[htb]
\begin{picture}(60,40)(0,2)
\pcline{->}(1,40)(14,40)\taput{$\ss'_1$}
\pcline{->}(16,40)(29,40)\taput{$\ss'_2$}
\pcline{->}(31,40)(44,40)\taput{$\ss'_3$}
\pcline{->}(46,40)(59,40)\taput{$\ss'_4$}
\pcline{->}(1,30)(14,30)
\pcline{->}(16,30)(29,30)
\pcline{->}(31,30)(44,30)
\pcline{->}(1,20)(14,20)
\pcline{->}(16,20)(29,20)
\pcline{->}(1,10)(14,10)
\psline[style=double](1,0)(8,0)
\psarc[style=double](8,7){7}{270}{360}
\psline[style=double](15,7)(15,9)
\psline[style=double](16,10)(23,10)
\psarc[style=double](23,17){7}{270}{360}
\psline[style=double](30,17)(30,19)
\psline[style=double](31,20)(38,20)
\psarc[style=double](38,27){7}{270}{360}
\psline[style=double](45,27)(45,29)
\psline[style=double](46,30)(53,30)
\psarc[style=double](53,37){7}{270}{360}
\psline[style=double](60,37)(60,39)
\pcline{->}(0,39)(0,31)\tlput{$\ss_1$}
\pcline{->}(0,29)(0,21)\tlput{$\ss_2$}
\pcline{->}(0,19)(0,11)\tlput{$\ss_3$}
\pcline{->}(0,9)(0,1)\tlput{$\ss_4$}
\pcline{->}(15,39)(15,31)
\pcline{->}(15,29)(15,21)
\pcline{->}(15,19)(15,11)
\pcline{->}(30,39)(30,31)
\pcline{->}(30,29)(30,21)
\pcline{->}(45,39)(45,31)
\psarc[style=thin](15,20){3}{180}{270}
\psarc[style=thin](15,30){3}{180}{270}
\psarc[style=thin](30,30){3}{180}{270}
\psarc[style=thin](15,40){3}{180}{270}
\psarc[style=thin](30,40){3}{180}{270}
\psarc[style=thin](45,40){3}{180}{270}
\put(2,16){\sf\small step 1:}\put(7,12){$\nmr_{\!3}$}
\put(2,26){\sf\small step 2:}\put(7,22){$\nmr_{\!2}$}
\put(17,26){\sf\small step 3:}\put(22,22){$\nmr_{\!3}$}
\put(2,36){\sf\small step 4:}\put(7,32){$\nmr_{\!1}$}
\put(17,36){\sf\small step 5:}\put(22,32){$\nmr_{\!2}$}
\put(32,36){\sf\small step 6:}\put(37,32){$\nmr_{\!3}$}
\end{picture}
\caption[]{\sf\smaller Normalising a length-$\pp$ word in $\pp(\pp-1)/2$~steps, here with $\pp = 4$:  according to~\eqref{E:NormClass3}, the six steps correspond to applying $\nmr_{\!\delta_4}$, that is, $\nmr_{\!323123}$.}
\label{F:Grid}
\end{figure}

An important consequence of Proposition~\ref{P:NormClass3} is that, contrary to the situation of Example~\ref{X:Large4Class}, it is impossible to have a large $4$-class in class~$(4, 3)$.

\begin{coro}\label{C:PClass}
If $(\SS, \nm)$ is a quadratic normalisation of class~$(4, 3)$, then $(\SS, \nm)$ is of $\pp$-class~$(4, 3)$ for every~$\pp \ge 3$.
\end{coro}

\begin{proof}
Assume $\pp \ge 3$, and let $\ww = \ss_1 \sep \pdots \sep \ss_\pp$ lie in~$\Pow\SS\pp$. We shall show that $\nm(\ww)$ is equal to $\nmR{\pp-1}_{212}(\ww)$ by checking that successively applying $\nmR{\pp-1}_2$, $\nmR{\pp-1}_1$, and $\nmR{\pp-1}_2$ to~$\ww$ leads to an $\nm$-normal word. First, let $\ss'_2 \sep \pdots \sep \ss'_\pp = \nm(\ss_2 \sep \pdots \sep \ss_\pp)$. We have 
\begin{equation}\label{E:PClass1}
\nmR{\pp-1}_2(\ww) = \ss_1 \sep \ss'_2 \sep \pdots \sep \ss'_\pp.
\end{equation}
As $\ss'_2 \sep \pdots \sep \ss'_\pp$ is $\nm$-normal, Lemma~\ref{L:LeftMult} implies $\nm(\ww) = \nmr_{\!12\pdots(\pp-1)}(\ss_1 \sep \ss'_2 \sep \pdots \sep \ss'_\pp)$. Now, put $\tt_0 = \ss_1$ and, inductively, $\ss''_\ii \sep \tt_\ii = \nmr(\tt_{\ii-1} \sep \ss'_{\ii+1})$ for $\ii = 1 \wdots \pp-1$. As $\ss'_2 \sep \pdots \sep \ss'_{\pp-1}$ is $\nm$-normal, Lemma~\ref{L:LeftMult} implies that $\ss''_1 \sep \pdots \sep \ss''_{\pp-2} \sep \tt_{\pp-1}$ is $\nm$-normal, so we have $\nm(\ss_1 \sep \ss'_2 \sep \pdots \sep \ss'_{\pp-1}) = \ss''_1 \sep \pdots \sep \ss''_{\pp-2} \sep \tt_{\pp-1}$, whence
\begin{equation}\label{E:PClass2}
\nmR{\pp-1}_1(\ss_1 \sep \ss'_2 \sep \pdots \sep \ss'_\pp) = \ss''_1 \sep \pdots \sep \ss''_{\pp-2} \sep \tt_{\pp-1} \sep \ss'_\pp.
\end{equation}
By the same argument, $\ss''_1 \sep \pdots \sep \ss''_{\pp-1} \sep \tt_\pp$ is $\nm$-normal, hence so is a fortiori $\ss''_2 \sep \pdots \sep \ss''_{\pp-1} \sep \tt_\pp$. By construction, we have 
$\ss''_2 \sep \pdots \sep \ss''_{\pp-1} \sep \tt_\pp = \nmr_{\!\pp-1}(\ss''_2 \sep \pdots \sep \ss''_{\pp-2} \sep \tt_{\pp-1} \sep \ss'_\pp)$, whence $\nm(\ss''_2 \sep \pdots \sep \ss''_{\pp-2} \sep \tt_{\pp-1} \sep \ss'_\pp) = \ss''_2 \sep \pdots \sep \ss''_{\pp-1} \sep \tt_\pp$, and, from there, 
\begin{equation}\label{E:PClass3}
\nmR{\pp-1}_2(\ss''_1 \sep \pdots \sep \ss''_{\pp-2} \sep \tt_{\pp-1} \sep \ss'_\pp) = \ss''_1 \sep \pdots \sep \ss''_{\pp-1} \sep \tt_\pp.
\end{equation}
As $\ss''_1 \sep \pdots \sep \ss''_{\pp-1} \sep \tt_\pp$ is $\nm$-normal, it is $\nm(\ww)$, so merging~\eqref{E:PClass1}, \eqref{E:PClass2}, and~\eqref{E:PClass3} gives $\nm(\ww) = \nmR{\pp-1}_{212}(\ww)$, witnessing that $(\SS, \nm)$ is of right-$\pp$-class~$3$.
\end{proof}

%%%%
\subsection{Axiomatisation}\label{SS:Axiomatisation}

By Proposition~\ref{P:Class}, if a quadratic normalisation $(\SS, \nm)$ is not of minimal left-class~$\infty$, the restriction~$\nmr$ of~$\nm$ to~$\Pow\SS2$ satisfies~\eqref{E:Class1} and its symmetric counterpart for~$\cc$ large enough. In particular, if $(\SS, \nm)$ has class~$(4, 3)$, then $\nmr$ satisfies~\eqref{E:Class3}, that is, $\nmr_{\!212} = \nmr_{\!2121} = \nmr_{\!1212}$. We now go in the other direction, and prove that any idempotent map~$\ft$ on~$\Pow\SS2$ satisfying the above relation necessarily stems from a quadratic normalisation of class~$(4,3)$, thus completing a proof of Theorem~A.

\begin{prop}\label{P:Axiom}
If~$\SS$ is a set and~$\ft$ is a map from~$\Pow\SS2$ to itself satisfying
\begin{equation}\label{E:Axiom}
\ft_{212} = \ft_{2121} = \ft_{1212},
\end{equation}
there exists a quadratic normalisation~$(\SS, \nm)$ of class~$(4, 3)$ satisfying~$\ft=\nmr$.
\end{prop}

The problem is to extend~$\ft$ into a map~$\ft^*$ on~$\SS^*$ such that~$(\SS,\ft^*)$ is a quadratic normalisation of class~$(4, 3)$. Proposition~\ref{P:NormClass3} leads us into introducing the following extension of~$\ft$.

\begin{defi}
For~$\ft$ a map from~$\Pow\SS2$ to itself, we write $\ft^*$ for the extension of~$\ft$ to~$\SS^*$ defined by $\ft^*(\ss) = \ss$ for~$\ss$ in~$\SS$ and $\ft^*(\ww) = \ft_{\delta_\pp}(\ww)$ for~$\ww$ in~$\Pow\SS\pp$.
\end{defi}

We will prove that, when \eqref{E:Axiom} is satisfied, $(\SS, \ft^*)$ is a quadratic normalisation of class~$(4, 3)$. We begin with preparatory formulas of the form $\ft_\uu = \ft_\vv$ when~$\uu$ and~$\vv$ are sequences of positions connected by a specific equivalence relation.

\begin{defi}
We denote by~$\equiv$ the congruence on the free monoid~$(\NNNN\setminus\{0\})^*$ of all finite sequences of positive integers generated by all shifted copies of
$$1 \sp 1 \equiv 1, \quad
1 \sp 2 \sp 1 \sp 2 \equiv 2 \sp 1 \sp 2 \sp 1 \equiv 2 \sp 1 \sp 2, \quad
1 \sp \ii \equiv \ii \sp 1 \text{\ for $\ii \ge 3$}.$$
\end{defi}

Note that the corresponding quotient monoid is a variant with infinitely many generators of the monoid~$\MM_{\nn}$ of~\cite{KraArt}. Lemmas~\ref{L:IterSym} and~\ref{L:Delta} below are essentially equivalent to~\cite[Prop.~67]{KraArt} but we include a short self-contained proof as our framework and notation are different.

\begin{lemm}\label{L:Equiv}
If~$\ft^*$ satisfies the conditions of Proposition~\ref{P:Axiom}, then, for all $\equiv$-equivalent sequences~$\uu$ and~$\vv$, we have $\ft_\uu = \ft_\vv$.
\end{lemm}

\begin{proof}
As~$\ft^*$ is idempotent, $\ft_{\ii\sp\ii}$ coincides with~$\ft_\ii$. For $\vert\ii-\jj\vert \ge 2$, $\ft_\ii$ and $\ft_\jj$ act on disjoint factors, so they commute. Finally, the relations for~$2\sp1\sp2$ and their shifted copies directly reflect~\eqref{E:Axiom}.
\end{proof}

\begin{lemm}\label{L:IterSym}
The following relations are valid for every $\pp \ge 2$:
\begin{gather}
\label{E:IterSym1}
\delta_\pp \equiv \pp-1 \sp \pdots \sp 2 \sp 1 \sp\sh(\delta_{\pp-1}),\\
\label{E:IterSym2}
\pp \sp \pdots \sp 2 \sp 1 \sp \pp \sp \pdots \sp 2 \sp 1 \equiv \pp \sp \pdots \sp 2 \sp 1 \sp \pp \sp \pdots \sp 3 \sp 2,\\
\label{E:IterSym3}
1 \sp 2 \sp \pdots \sp \pp \sp 1 \sp 2 \sp \pdots \sp \pp \equiv 2 \sp 3 \sp \pdots \pp \sp 1 \sp 2 \sp \pdots \sp \pp.
\end{gather}
\end{lemm}

\begin{proof}
We use induction on~$\pp \ge 2$. For $\pp = 2$, \eqref{E:IterSym1} reads $1 \equiv 1 \sp \sh(\varepsilon)$, which is valid. Assume $\pp \ge 3$. Then we find
\begin{align*}
\delta_\pp 
&= \sh(\delta_{\pp-1}) \sp 1 \sp 2 \sp \pdots \sp \pp-1 
	&&\text{by definition of~$\delta_\pp$,}\\
&\equiv \sh(\pp-2 \sp \pdots \sp 2 \sp 1 \sp \sh(\delta_{\pp-2})) \sp 1 \sp 2 \sp \pdots \sp \pp-1 
	&&\text{by induction hypothesis,}\\
&= \pp-1 \sp \pdots \sp 3 \sp 2 \sp \sh^2(\delta_{\pp-2})) \sp 1 \sp 2 \sp \pdots \sp \pp-1\\
&\equiv \pp-1 \sp \pdots \sp 3 \sp 2 \sp 1 \sp \sh^2(\delta_{\pp-2})) \sp 2 \sp \pdots \sp \pp-1
	&&\text{by $\ii\sp 1\equiv 1\sp\ii$ for~$\ii=\sh^2(\jj)$,}\\
&= \pp-1 \sp \pdots \sp 3 \sp 2 \sp 1 \sp \sh(\sh(\delta_{\pp-2})) \sp 1 \sp \pdots \sp \pp-2)\\
&= \pp-1 \sp \pdots \sp 3 \sp 2 \sp 1 \sp \sh(\delta_{\pp-1})
	&&\text{by definition of~$\delta_{\pp-1}$.}
\end{align*}
Next, for $\pp = 2$, \eqref{E:IterSym2} reads $2\sp 1 \sp 2 \sp 1 \equiv 2 \sp 1 \sp 2$, which is valid. For $\pp \ge 3$, we find
\begin{align*}
\pp \sp \pdots \sp 2 \sp 1 \sp \pp \sp \pdots \sp 2 \sp 1 
&\equiv \pp \sp \pdots \sp 2 \sp \pp \sp \pdots \sp 3 \sp 1 \sp 2 \sp 1
	&&\text{by $1\sp\ii \equiv \ii\sp 1$ for~$\ii\ge 3$,} \\
&= \sh(\pp-1 \sp \pdots \sp 1 \sp \pp-1 \sp \pdots \sp 2) \sp 1 \sp 2 \sp 1\\
&\equiv \sh(\pp-1 \sp \pdots \sp 1 \sp \pp-1 \sp \pdots \sp 2 \sp 1) \sp 1 \sp 2 \sp 1
	&&\text{by induction hypothesis,}\\
&= \pp \sp \pdots \sp 2 \sp \pp \sp \pdots \sp 3 \sp 2 \sp 1 \sp 2 \sp 1\\
&\equiv \pp \sp \pdots \sp 2 \sp \pp \sp \pdots \sp 3 \sp 2 \sp 1 \sp 2
	&&\text{by definition of~$\equiv$,}\\
&= \sh(\pp-1 \sp \pdots \sp 1 \sp \pp-1 \sp \pdots \sp 2 \sp 1) \sp 1 \sp 2\\
&\equiv \sh(\pp-1 \sp \pdots \sp 1 \sp \pp-1 \sp \pdots \sp 2) \sp 1 \sp 2
	&&\text{by induction hypothesis,}\\
&= \pp \sp \pdots \sp 2 \sp \pp \sp \pdots \sp 3 \sp 1 \sp 2 \\
&\equiv \pp \sp \pdots \sp 2 \sp 1 \sp \pp \sp \pdots \sp 3 \sp 2
	&&\text{by $\ii \sp 1 \equiv 1\sp \ii$ for~$\ii\ge 3$.}
\end{align*}
The argument for \eqref{E:IterSym3} is symmetric.
\end{proof}

\begin{lemm}\label{L:Delta}
The following relations are valid for every $\pp \ge 2$ and $1 \le \ii < \pp$:
\begin{equation}\label{E:Delta}
\delta_\pp \equiv \delta_\pp \sp \ii \equiv \ii \sp \delta_\pp.
\end{equation}
\end{lemm}

\begin{proof}
We use induction on $\pp \ge 2$. For $\pp = 2$, the only case to consider is $\ii = 1$, and \eqref{E:Delta} then reduces to $1 \equiv 1 \sp 1$, which is true by definition. Assume $\pp = 3$. For $\ii = 1$, \eqref{E:Delta} reduces to $2 \sp 1 \sp 2 \equiv 2 \sp 1 \sp 2 \sp 1 \equiv 1 \sp 2 \sp 1 \sp 2$ and, for $\ii = 2$, to $2 \sp 1 \sp 2 \equiv 2 \sp 1 \sp 2 \sp 2$ and $2 \sp 1 \sp 2 \equiv 2 \sp 2 \sp 1 \sp2$, which are true by the definition of~$\equiv$.

Assume now $\pp \ge 4$. For $\ii = 1$, we find
\begin{align*}
\delta_\pp \sp 1
&\equiv \pp-1 \sp \pdots \sp 2 \sp 1 \sp \pp-1 \sp \pdots \sp 2 \sp \sh^2(\delta_{\pp-2}) \sp 1 
&&\text{by~\eqref{E:IterSym1} twice,} \\
&\equiv \pp-1 \sp \pdots \sp 2 \sp 1 \sp \pp-1 \sp \pdots \sp 2 \sp 1 \sp \sh^2(\delta_{\pp-2})
&&\text{by $\ii\sp 1\equiv 1\sp\ii$ for $\ii\ge 3$,} \\
&\equiv \pp-1 \sp \pdots \sp 2 \sp 1 \sp \pp-1 \sp \pdots \sp 2 \sp \sh^2(\delta_{\pp-2})
&&\text{by~\eqref{E:IterSym2},} \\
&\equiv \pp-1 \sp \pdots \sp 2 \sp 1 \sp \sh(\delta_{\pp-1}) \equiv \delta_\pp
&&\text{by~\eqref{E:IterSym1} twice,}
\end{align*}
and, for $2 \le \ii < \pp$, using~\eqref{E:IterSym1} and the induction hypothesis, we find
$$\delta_\pp \sp \ii
\equiv \pp-1 \sp \pdots \sp 1 \sp \sh(\delta_{\pp-1}) \sp \ii
= \pp-1 \sp \pdots \sp 1 \sp \sh(\delta_{\pp-1} \sp \ii-1)
\equiv \pp-1 \sp \pdots \sp 1 \sp \sh(\delta_{\pp-1}) \equiv \delta_\pp.$$
The argument for $\ii \sp \delta_\pp \equiv \delta_\pp$ is symmetric, with the definition of~$\delta_\pp$ and~\eqref{E:IterSym3} replacing~\eqref{E:IterSym1} and~\eqref{E:IterSym2}.
\end{proof}

\begin{proof}[Proof of Proposition~\ref{P:Axiom}]
Assume that~$\ft$ is idempotent and satisfies~\eqref{E:Axiom}. Lemmas~\ref{L:Equiv} and~\ref{E:Delta} imply, for all~$\pp \ge 3$ and every~$\ww$ in~$\Pow\SS\pp$, the equalities 
\begin{equation}\label{E:Method}
\ft^*(\ft_\ii(\ww)) = \ft^*(\ww) \quad \text{and} \quad \ft_\ii(\ft^*(\ww)) = \ft^*(\ww) \text{\quad for $\ii$ with $1 \le \ii < \pp$}.
\end{equation}
Let $\uu,\vv,\ww$ be $\SS$-words with respective lengths~$\mm,\nn,\pp$. Then we find
$$\ft^*(\uu\sep\ft^*(\ww)\sep\vv) = \ft_{\delta_{\mm+\nn+\pp}}(\uu\sep\ft_{\delta_{\pp}}(\ww)\sep\vv) = \ft_{\delta_{\mm+\nn+\pp}}(\ft_{\sh^{\mm}(\delta_{\pp})}(\uu\sep\ww\sep\vv)).$$
The map~$\ft_{\sh^{\mm}(\delta_{\pp})}$ is a composite of maps~$\ft_{\ii}$ with~$\mm+1\le\ii\le\mm+\pp-1$, so Lemma~\ref{L:Delta} implies
$$\ft^*(\uu\sep\ft^*(\ww)\sep\vv) = \ft_{\delta_{\mm+\nn+\pp}}(\uu\sep\ww\sep\vv) = \ft^*(\uu\sep\ww\sep\vv),$$
and we deduce that $(\SS,\ft^*)$ satisfies~\eqref{E:NormSys3}. As it also satisfies~\eqref{E:NormSys1} and~\eqref{E:NormSys2} by the definition of~$\ft^*$, it is a normalisation.
\end{proof}

The following example shows that the axiomatisation of class~$(4,3)$ normalisations provided by Proposition~\ref{P:Axiom} does not extend to higher classes.

\begin{exam}
Let us consider the rewriting system of Example~\ref{X:Large4Class} with $\nn=\infty$, that is, $\SS_{\infty}=\{ \tta, \ttb_1, \ttb_2, ..., \ttc_1, \ttc_2, ...\}$ with the rules $\tta\ttb_\ii {\to} \tta\ttb_{\ii+1}$ and $\ttb_{\ii+1}\ttc_\ii {\to} \ttb_{\ii+1} \ttc_{\ii+1}$ for~$\ii$ odd and $\ttc_\ii\tta {\to} \ttc_{\ii+1}\tta$ and $\ttb_\ii\ttc_{\ii+1} {\to} \ttb_{\ii+1}\ttc_{\ii+1}$ for~$\ii$ even. The associated map~$\ft$ on~$\Pow{\SS_{\infty}}2$ satisfies the relation $\ft_{12121}=\ft_{21212}$, but no quadratic normalisation~$(\SS_{\infty},\nm)$ satisfies $\ft=\nmr$. Indeed, no $\SS_{\infty}$-word that can be reached from $\tta\ttb_1\ttc_1\tta$ by successive applications of~$\ft$ on length-two factors is normal.
\end{exam}

%%%%
\section{Class and termination}\label{S:Termination}

By Proposition~\ref{P:QuadNormRewr}, a quadratic normalisation~$(\SS,\nm)$ yields a reduced quadratic rewriting system~$(\SS,\RR)$ that is normalising and confluent. This however does not rule out the possible existence of infinite rewriting sequences, which we investigate here.

The section comprises three parts. We first consider the case of class~$(3, 3)$ and prove an easy convergence result (Subsection~\ref{SS:Termin33}). Next, the case of classes~$(3,4)$ and~$(4,3)$ is investigated in Subsection~\ref{SS:Termin43}, where the not-so-easy convergence result stated as Theorem~B is established. Finally, we show in Subsection~\ref{SS:Termin44} that the previous result is optimal by constructing a non-convergent example in class~$(4, 4)$.

%%%%
\subsection{Termination in class~$(3,3)$}\label{SS:Termin33}

We first consider the case of quadratic normalisations of class~$(3,3)$, and we use an argument of finiteness on symmetric groups to prove:

\begin{prop}\label{P:Termin33}
If $(\SS, \nm)$ is a quadratic normalisation of class~$(3,3)$, then the associated rewriting system~$(\SS, \RR)$ is convergent, and so is $(\SSm,\RR_\ee)$ if~$\ee$ is an $\nm$-neutral element of~$\SS$. More precisely, every rewriting sequence from a length-$\pp$ word has length at most $\pp(\pp-1)/2$.
\end{prop}

\begin{proof}
Assume that $\ww_0\wdots\ww_{\ell}$ are $\SS$-words satisfying $\ww_{\kk}\fl_\RR\ww_{\kk+1}$ for $0\le\kk<\ell$. Let~$\pp$ be the common length of the $\SS$-words~$\ww_\kk$. By assumption, for every~$\kk \ge 1$, we have $\ww_{\kk}=\nmr_{\!\ii_{\kk}}(\ww_{\kk-1})$ for some~$1\le \ii_{\kk}<\pp$, with $\ww_\kk \neq \ww_{\kk-1}$. 

Let us observe, by induction on~$\kk$, that~$\uu=\ii_1\sep\pdots\sep\ii_\kk$ is reduced in the sense of Coxeter theory, that is, it is a minimal-length representative of the associated element of the symmetric group~$\Sym_{\pp}$. For $\kk = 0$, the result is true as the empty word is reduced. Assume $\kk \ge 1$, and write $\uu = \uu'\sep \ii$. By induction hypothesis,~$\uu'$ is reduced. If $\uu'\sep\ii$ is not reduced, then, by the exchange lemma for~$\Sym_\pp$, see~\cite{BjBr}, there exists a sequence of positions~$\uu''$ such that~$\uu'$ is equivalent to~$\uu''\sep\ii$ modulo the braid relations. Now, by Proposition~\ref{P:Class}, the assumption that~$(\SS, \nm)$ is of class~$(3, 3)$ implies $\nmr_{\!121} = \nmr_{\!212}$ and, from there, the equivalence of~$\uu'$ and~$\uu''\sep\ii$ modulo the braid relations implies $\nmr_{\!\uu'} = \nmr_{\!\uu''\sep\ii}$. Putting $\ww'=\nmr_{\!\uu''}(\ww_0)$, we obtain $\ww_{\kk-1}=\nmr_{\!\ii}(\ww')$ and, since~$\nmr$ is idempotent, $\ww_{\kk}=\nmr_{\!\ii}(\nmr_{\!\ii}(\ww'))=\nmr_{\!\ii}(\ww')=\ww_{\kk-1}$, which contradicts $\ww_{\kk-1}\fl_{\RR}\ww_{\kk}$. So~$\uu$ must be reduced.

Now, it is well-known that the length~$\ell$ of a reduced word representing an element of~$\Sym_\pp$ is bounded above by $\pp(\pp-1)/2$, for instance because~$\ell$ is the number of inversions of the permutation represented by~$\uu$. So~$(\SS,\RR)$ terminates and, by Proposition~\ref{P:NeutralQuadRewr}, so does~$(\SS_\ee,\RR_\ee)$ if~$\ee$ is an $\nm$-neutral element in~$\SS$.
\end{proof}

\begin{rema*}
The bound $\pp(\pp-1)/2$ in Proposition~\ref{P:Termin33} is sharp, since, for the lexicographic normalisation~$(\SS,\nm)$ of Example~\ref{X:Abelian}, normalising $\tta_{\pp}\sep\pdots\sep\tta_1$ with $\tta_1<\pdots<\tta_{\pp}$ actually requires $\pp(\pp- 1)/2$ steps. 
\end{rema*}

Proposition~\ref{P:Termin33} applies to the example of plactic monoids, described thereafter. Those monoids have known normalisations that fit into our setting of quadratic normalisations, and were among the original motivations for extending the framework of Garside normalisation to the current one.

\begin{exam}\label{X:Plactic}
If~$\XX$ is a totally ordered finite set, the \emph{plactic monoid over~$\XX$} is the monoid~$\PP_{\XX}$ generated by~$\XX$ and subject to $\ttx\ttz\tty = \ttz\ttx\tty$, for $\ttx\le\tty<\ttz$, and $\tty\ttx\ttz = \tty\ttz\ttx$, for $\ttx<\tty\le\ttz$. We refer to~\cite{CainGrayMalheiro} for a recent reference on the following facts. The monoid~$\PP_{\XX}$ is also generated by the family~$\SS$ of columns over~$\XX$ (the strictly decreasing products of elements of~$\XX$). A pair $\cc\sep\cc'$ of columns is \emph{normal} if $\LG{\cc}\ge\LG{\cc'}$ holds and, for every $1\le\kk\le\LG{\cc'}$, the $\kk$th element of~$\cc$ is at most the one of~$\cc'$. Every equivalence class of $\XX$-words contains a unique tableau (a product $\cc_1\pdots\cc_{\nn}$ of columns such that each~$\cc_{\ii}\sep\cc_{\ii+1}$ is normal), with minimal length in terms of columns: thus, mapping a $\SS$-word to the unique corresponding tableau defines a geodesic normal form~$\NF$ on~$(\PP_{\XX},\SS_{\ee})$, where~$\ee$ denotes the empty column. 

We consider the normalisation~$(\SS^{\ee},\nm)$ associated to~$\NF$, which satisfies~\eqref{E:Quad1} by the definition of tableaux. Moreover, for every $\SS$-word~$\ww$, the tableau~$\NF(\ww)$ can be computed from any $\SS$-word~$\ww$ by Schensted's insertion algorithm, progressively replacing each pair $\cc\sep\cc'$ of subsequent columns of~$\ww$ by $\NF(\cc\sep\cc')$, which is a tableau with one or two columns. So, the normalisation~$(\SS^{\ee},\nm)$ also satisfies~\eqref{E:Quad2}, so that it is quadratic, and, when~$\XX$ contains at least two elements, it is of minimal class~$(3,3)$ as testified by the computations of~\cite[\S\S4.2--4.4]{BokutChenChenLi}. By Proposition~\ref{P:Termin33}, we recover~\cite[Th.~3.4]{CainGrayMalheiro}: the rewriting system~$(\SS,\RR)$ with $\RR=\{ \cc\sep\cc' \fl \NF(\cc\sep\cc') \mid \cc,\cc'\in\SS\}$ is finite, convergent and it presents~$\PP_{\XX}$. A similar argument leads to a (nonfinite) convergent quadratic presentation of~$\PP_{\XX}$ in terms of \emph{rows}, which are nondecreasing products of elements of~$\XX$. The proof that the class is~$(3,3)$ is given in~\cite[\S\S3.2--3.4]{BokutChenChenLi}.
\end{exam}

%%%%
\subsection{Termination in class~$(4,3)$}\label{SS:Termin43}

We now consider the case of class~$(4, 3)$ and establish the general termination result stated as Theorem~B:

\begin{prop}\label{P:43Terminating}
If $(\SS, \nm)$ is a quadratic normalisation of class~$(4,3)$, then the associated rewriting system~$(\SS, \RR)$ is convergent, and so is $(\SSm,\RR_\ee)$ if~$\ee$ is an $\nm$-neutral element of~$\SS$. More precisely, every rewriting sequence from a length-$\pp$ word has length at most $2^\pp - \pp - 1$.
\end{prop}

Proposition~\ref{P:43Terminating} subsumes Proposition~\ref{P:Termin33}. But its proof resorts to different arguments, since Krammer's monoid~$\MM_\pp$, see~\cite{KraArt}, which should replace here the finite quotient-monoid $\BP\pp{/}{\{\sigg\ii2 = \sig\ii \mid 1 \le \ii < \pp\}}$, with $121 = 212$ substituted with $121 = 2121$, is infinite. Instead, we analyse $\nm$-normalisation directly to show that no infinite rewriting sequence may exist because one inevitably proceeds to the normal form.

\begin{proof}
Let $\FF(\pp)$ denote the maximal length of sequences $\ww_0\fl_\RR\ww_1\fl_\RR\pdots\fl_\RR\ww_{\ell}$ of $\SS$-words of length~$\pp$, possibly~$\infty$. We prove the inequality $\FF(\pp) \le 2^\pp - \pp - 1$ using induction on $\pp \ge 2$. For $\pp = 2$, the inequality $\FF(\pp) \le 1$ holds, since~$\nm$ is idempotent. We now assume $\pp \ge 3$ and consider a sequence $\www = (\ww_0\wdots\ww_{\ell})$ of length-$\pp$ words satisfying $\ww_{\kk}\fl_\RR\ww_{\kk+1}$ for $0\le\kk<\ell$. We shall distinguish several types of rewriting steps in the sequence~$\www$, in connection with Proposition~\ref{P:NormClass3} and the triangular grid diagram of Figure~\ref{F:Grid}. The latter corresponds to an optimal strategy, which needs not be the case for~$\www$, but we shall explain how to enrich each word~$\ww_\kk$ into a word~$\wwh_\kk$ by attaching with each letter of~$\ww$ a direction, either horizontal or vertical. We define $\widehat{\SS}$ as~$\SS \amalg \overline{\SS}$, where~$\overline\SS$ is a copy of~$\SS$ with an element~$\overline\ss$ for each~$\ss$ in~$\SS$, and we take the convention that~$\ss$ means ``vertical~$\ss$'' and~$\overline\ss$ means ``horizontal~$\ss$'': this associates with every $\widehat\SS$-word~$\widehat\ww$ a path in a triangular grid by starting from the top-left corner and attaching to the successive letters of~$\widehat\ww$ horizontal left-to-right edges and vertical top-to-down edges. 

We construct the $\widehat{\SS}$-words~$\wwh_\kk$ inductively, in such a way that
\begin{equation}\label{E:Cond}
\text{for every length-two factor $\overline\ss \sep \tt$ or $\overline\ss \sep \overline\tt$ of~$\wwh_\kk$, the $\SS$-word $\ss\sep\tt$ is $\nm$-normal.}
\end{equation}
First, for $\ww_0 = \ss_1 \sep \pdots \sep \ss_\pp$, we put $\wwh_0 = \ss_1 \sep \pdots \sep \ss_{\pp-1} \sep \overline{\ss_\pp}$. Then $\wwh_0$ satisfies~\eqref{E:Cond} by default, and (the path associated with)~$\wwh_0$ consists of~$\ww_0$ drawn vertically, except the last letter, which is drawn horizontally. 

Assume that~$\wwh_{\kk-1}$ has been defined, it satisfies~\eqref{E:Cond}, and $\ww_\kk=\nmr_{\!\ii}(\ww_{\kk-1})$ holds. Let~$\ss$ and~$\tt$ be the letters of~$\ww_{\kk-1}$ in positions~$\ii$ and~$\ii+1$, and $\ss' \sep \tt' = \nm(\ss \sep \tt)$. We look at the directions of the letters of~$\wwh_{\kk-1}$ in positions~$\ii$ and~$\ii+1$. The assumption $\ww_{\kk-1}\fl_\RR\ww_{\kk}$ implies that $\ss \sep \tt$ is not $\nm$-normal. By~\eqref{E:Cond}, this excludes the directions $\overline\ss \sep \tt$ and $\overline\ss \sep \overline\tt$. So only two cases are possible.

In the case of a \emph{VH-step}, meaning a vertical letter followed by a horizontal one, we define~$\wwh_\kk$ to be obtained from~$\wwh_{\kk-1}$ by replacing, at position~$\ii$, the factor $\ss \sep \overline\tt$ with $\overline{\ss'} \sep \tt'$ when $\tt$ is not the last letter of~$\ww_{\kk-1}$, and by $\overline{\ss'} \sep \overline{\tt'}$ otherwise, which corresponds to replacing 
\VR(5,3)\begin{picture}(7,0)(-1,2)
\pcline{->}(0,6)(0,0)\put(0.5,3.5){$\ss$}
\pcline{->}(0,0)(6,0)\put(3,1){$\tt$}
\end{picture}
with 
\VR(5,3)\begin{picture}(10,0)(-1,2)
\pcline{->}(0,6)(6,6)\put(0.5,3){$\ss'$}
\pcline{->}(6,6)(6,0)\put(6.5,2.5){$\tt'$}
\psarc[style=thin](6,6){2}{180}{270}
\end{picture} or
\VR(4,0)\begin{picture}(14,0)(-1,0)
\pcline{->}(0,0)(6,0)\put(1,1){$\ss'$}
\pcline{->}(6,0)(12,0)\put(9,1){$\tt'$}
\psarc[style=thin](5.5,0){1.5}{0}{180}
\end{picture} respectively. In both cases, the length-two factor of~$\ww_{\kk}$ starting at position~$\ii$ is thus $\nm$-normal, so~\eqref{E:Cond} is satisfied at this position. The only other position where~\eqref{E:Cond} might fail is~$\ii-1$, when the corresponding letter of~$\wwh_{\kk}$ is horizontal, since, otherwise, \eqref{E:Cond} requires nothing on the factor. Now, going from~$\wwh_{\kk-1}$ to~$\wwh_{\kk}$ replaces 
\VR(5,2)\begin{picture}(14,0)(-1,2)
\pcline{->}(0,6)(6,6)
\pcline{->}(6,6)(6,0)
\pcline{->}(6,0)(12,0)
\end{picture}
with 
\VR(5,4)\begin{picture}(15,0)(-1,2)
\pcline{->}(0,6)(6,6)
\pcline[style=exist]{->}(6,6)(6,0)
\pcline[style=exist]{->}(6,0)(12,0)
\pcline{->}(6,6)(12,6)
\pcline{->}(12,6)(12,0)
\psarc[style=thin](12,6){2}{180}{270}
\end{picture}. 
But, by construction, the pattern 
\VR(5,4)\begin{picture}(14,0)(-1,2)
\pcline{->}(0,6)(6,6)
\pcline{->}(6,6)(6,0)
\pcline{->}(6,0)(12,0)
\end{picture}
necessarily comes from an earlier diagram
\VR(5,4)\begin{picture}(13,0)(-1,2)
\pcline{->}(0,6)(6,6)
\pcline{->}(0,6)(0,0)
\pcline{->}(0,0)(6,0)
\pcline{->}(6,6)(6,0)
\pcline{->}(6,0)(12,0)
\psarc[style=thin](6,6){2}{180}{270}
\psarc[style=thin](5.5,0){1.5}{180}{360}
\end{picture}, 
in which the pairs indicated with small arcs are $\nm$-normal by induction hypothesis. Hence, going to~$\ww_\kk$ means going to
\VR(6,4)\begin{picture}(14,0)(-1,2)
\pcline{->}(0,6)(6,6)
\pcline{->}(0,6)(0,0)
\pcline{->}(0,0)(6,0)
\pcline{->}(6,6)(6,0)
\pcline{->}(6,0)(12,0)
\pcline{->}(6,6)(12,6)
\pcline{->}(12,6)(12,0)
\psarc[style=thin](6,6){2}{180}{270}
\psarc[style=thin](5.5,0){1.5}{180}{360}
\psarc[style=thin](12,6){2}{180}{270}
\psarc[style=thinexist](5.5,6){1.5}{0}{180}
\end{picture}
and the domino rule precisely implies that the top two horizontal edges form an $\nm$-normal word. So $\wwh_\kk$ satisfies~\eqref{E:Cond}.

In the case of a \emph{VV-step} (two vertical letters), we define $\wwh_\kk$ to be obtained from~$\wwh_{\kk-1}$ by replacing, at position~$\ii$, the factor $\ss \sep \tt$ with $\ss' \sep \tt'$. As the shape of~$\wwh_\kk$ is the same as the one of~$\wwh_{\kk-1}$, we only have to check~\eqref{E:Cond} for the length-two factor at position~$\ii-1$, and only when its first letter is horizontal, that is, one goes from
\VR(9,7)\begin{picture}(11,0)(-1,5)
\pcline{->}(0,12)(6,12)
\put(2,9.6){$\rr$}
\pcline{->}(6,12)(6,6)\trput{$\ss$}
\pcline{->}(6,6)(6,0)\trput{$\tt$}
\end{picture}
to
\VR(9,7)\begin{picture}(10,0)(0,5)
\pcline{->}(0,12)(6,12)
\put(2,9.6){$\rr$}
\pcline{->}(6,12)(6,6)\trput{$\ss'$}
\pcline{->}(6,6)(6,0)\trput{$\tt'$}
\psarc[style=thin](6,6.5){1.5}{270}{90}
\end{picture}.
But, by construction, the original pattern in~$\wwh_{\kk-1}$ arises from an earlier pattern
\VR(6,7)\begin{picture}(10,0)(-1,5)
\put(2,9.6){$\rr$}
\pcline{->}(0,12)(6,12)
\pcline{->}(0,6)(6,6)
\pcline{->}(0,0)(6,0)
\pcline{->}(0,12)(0,6)
\pcline{->}(0,6)(0,0)
\pcline{->}(6,12)(6,6)\trput{$\ss$}
\pcline{->}(6,6)(6,0)\trput{$\tt$}
\psarc[style=thin](6,12){2}{180}{270}
\psarc[style=thin](6,6){2}{180}{270}
\end{picture},
so that, when $\ss \sep \tt$ is replaced by $\ss' \sep \tt'$, the domino rule implies that $\rr \sep \ss'$ is $\nm$-normal, as the diagram 
\VR(6,6)\begin{picture}(15,0)(-1,5)
\pcline{->}(0,12)(6,12)\taput{$\rr$}
\pcline{->}(0,6)(6,6)
\pcline{->}(0,0)(6,0)
\pcline{->}(0,12)(0,6)
\pcline{->}(0,6)(0,0)
\pcline{->}(6,12)(6,6)\trput{$\ss$}
\pcline{->}(6,6)(6,0)\trput{$\tt$}
\psarc[style=thin](6,12){2}{180}{270}
\psarc[style=thin](5.5,6){1.5}{180}{360}
\pcline{->}(6,12)(12,12)\taput{$\ss'$}
\pcline{->}(6,6)(12,6)
\pcline{->}(12,12)(12,6)\trput{$\tt'$}
\psarc[style=thin](12,12){2}{180}{270}
\psline[style=double](6,0)(9,0)
\psarc[style=double](9,3){3}{270}{360}
\psline[style=double](12,3)(12,6)
\psarc[style=thinexist](5.5,12){1.5}{0}{180}
\end{picture}
witnesses.

The construction of $\wwh_0 \wdots \wwh_\ell$ is complete, and we now count how many VH-steps and VV-steps can occur in~$\www$. First, each $\widehat{\SS}$-word~$\wwh_\kk$ is associated with a path in the triangular grid diagram of Figure~\ref{F:Grid} and each VH-step causes this path to cross one square in the grid. As the latter contains~$\pp(\pp-1)/2$ squares, we deduce that there are at most $\pp(\pp-1)/2$ VH-steps in~$\www$.

We turn to VV-steps, partitioning them into several subtypes according to where they occur: we say that a VV-step is a $\VVV\jj$-step if it is located on the~$\jj$th column, that is, it replaces a vertical factor~$\ss\sep\tt$ of~$\wwh_\kk$ that is preceded by~$\jj-1$ horizontal letters. Now we fix~$\jj$ with $1 \le \jj < \pp$ and count the $\VVV\jj$-steps that can occur in~$\www$. For $1 \le \ii \le \pp-\jj$, let~$\ss_{\ii,\kk}$ be the letter that vertically occurs at the~$\ii$th position in the $\jj$th column in~$\wwh_\kk$, if it exists. For a given value of~$\ii$, define~$\ssm\ii$ to be~$\ss_{\ii,\kk}$ where~$\kk$ is minimal such that~$\ss_{\ii,\kk}$ exists (if any) and, symmetrically, let~$\ssp\ii$ be~$\ss_{\ii,\kk}$ where~$\kk$ is maximal such that~$\ss_{\ii,\kk}$ exists (if any). For each~$\kk$, if $\ssm\ii$ is defined for $\aa \le \ii $, if~$\ss_{\ii,\kk}$ is defined for $\bb \le \ii \le \cc$, and if~$\ssp\ii$ is defined for $\ii \le \dd$, we put 
$$\vv_{\kk} = \ssm\aa \sep\pdots\sep \ssm{\bb-1} \sep \ss_{\bb,\kk} \sep\pdots\sep \ss_{\cc,\kk} \sep \ssp{\cc+1} \sep\pdots\sep \ssp\dd.$$
So~$\vv_\kk$ is the factor of~$\ww_\kk$ forming the~$\jj$th column of~$\wwh_\kk$, preceded by the letters that are the first to appear in positions~$\aa\wdots\bb-1$ in column~$\jj$ in~$\wwh_{\kk+1}\wdots\wwh_{\ell}$, and followed by the last letters to appear in positions~$\cc+1\wdots\dd$ in column~$\jj$ in~$\wwh_0\wdots\wwh_{\kk-1}$. If one goes from~$\ww_{\kk-1}$ to~$\ww_\kk$ by a VH-step or a $\VVV{\jj'}$-step with $\jj' \not= \jj$, then we have $\vv_{\kk}=\vv_{\kk-1}$: indeed, either the~$\jj$th columns of~$\wwh_{\kk-1}$ and~$\wwh_{\kk}$ are equal, or a VH-step normalises the last letter~$\ss_{\cc,\kk-1}$ of the~$\jj$th column of~$\wwh_{\kk-1}$ with the subsequent horizontal letter, or the first letter~$\ss_{\bb,\kk-1}$ of the~$\jj$th column of~$\wwh_{\kk-1}$ with the previous horizontal letter; and, in the case of the last letter (the other one being symmetric), $\vv_{\kk}$ is~$\vv_{\kk-1}$ with~$\ss_{\cc,\kk}$ replaced by~$\ssp\cc$, hence unchanged by definition of~$\ssp\cc$. Otherwise, if one goes from~$\ww_{\kk-1}$ to~$\ww_\kk$ by a $\VVV\jj$-step, then $\vv_{\kk-1} \to_{\RR} \vv_{\kk}$ holds. As, by construction, the length of the $\SS$-word~$\vv_{\kk}$ is at most $\pp - \jj$, we conclude that the number of $\VVV\jj$-steps in~$\www$ is at most $\FF(\pp - \jj)$. Summing up, we deduce
\begin{equation}\label{E:Function}
\FF(\pp) \le \frac{\pp(\pp-1)}2 + \FF(\pp-1) + \pdots + \FF(3) + \FF(2),
\end{equation}
which solves into $\FF(\pp) \le 2^\pp + \pp-1$ owing to $\FF(\qq)\le 2^{\qq}-\qq-1$ that holds for every $2\le\qq<\pp$ by induction hypothesis.

Finally, as in the case of class~$(3, 3)$, Proposition~\ref{P:NeutralQuadRewr} implies that, if~$(\SS,\RR)$ terminates and~$\ee$ is an $\nm$-neutral element in~$\SS$, then so does~$(\SS_\ee,\RR_\ee)$.
\end{proof}

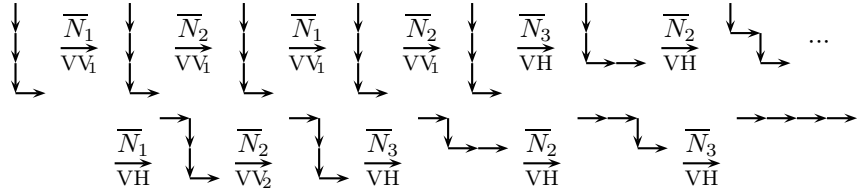
\begin{figure}[htb]
\begin{picture}(14,10)(0,0)
\pcline{->}(0,12)(0,8)
\pcline{->}(0,8)(0,4)
\pcline{->}(0,4)(0,0)
\pcline{->}(0,0)(4,0)
\pcline{->}(6,6)(11,6)\taput{$\nmr_{\!1}$}\tbput{\footnotesize$\VVV1$}
\end{picture}
\begin{picture}(14,10)(0,0)
\pcline{->}(0,12)(0,8)
\pcline{->}(0,8)(0,4)
\pcline{->}(0,4)(0,0)
\pcline{->}(0,0)(4,0)
\pcline{->}(6,6)(11,6)\taput{$\nmr_{\!2}$}\tbput{\footnotesize$\VVV1$}
\end{picture}
\begin{picture}(14,10)(0,0)
\pcline{->}(0,12)(0,8)
\pcline{->}(0,8)(0,4)
\pcline{->}(0,4)(0,0)
\pcline{->}(0,0)(4,0)
\pcline{->}(6,6)(11,6)\taput{$\nmr_{\!1}$}\tbput{\footnotesize$\VVV1$}
\end{picture}
\begin{picture}(14,10)(0,0)
\pcline{->}(0,12)(0,8)
\pcline{->}(0,8)(0,4)
\pcline{->}(0,4)(0,0)
\pcline{->}(0,0)(4,0)
\pcline{->}(6,6)(11,6)\taput{$\nmr_{\!2}$}\tbput{\footnotesize$\VVV1$}
\end{picture}
\begin{picture}(14,10)(0,0)
\pcline{->}(0,12)(0,8)
\pcline{->}(0,8)(0,4)
\pcline{->}(0,4)(0,0)
\pcline{->}(0,0)(4,0)
\pcline{->}(6,6)(11,6)\taput{$\nmr_{\!3}$}\tbput{\footnotesize VH}
\end{picture}
\begin{picture}(18,10)(0,0)
\pcline{->}(0,12)(0,8)
\pcline{->}(0,8)(0,4)
\pcline{->}(0,4)(4,4)
\pcline{->}(4,4)(8,4)
\pcline{->}(10,6)(15,6)\taput{$\nmr_{\!2}$}\tbput{\footnotesize VH}
\end{picture}
\begin{picture}(20,10)(0,0)
\pcline{->}(0,12)(0,8)
\pcline{->}(0,8)(4,8)
\pcline{->}(4,8)(4,4)
\pcline{->}(4,4)(8,4)
\put(10,6){$\pdots$}
\end{picture}\\
\begin{picture}(16,11)(0,4)
\pcline{->}(-6,6)(-1,6)\taput{$\nmr_{\!1}$}\tbput{\footnotesize VH}
\pcline{->}(0,12)(4,12)
\pcline{->}(4,12)(4,8)
\pcline{->}(4,8)(4,4)
\pcline{->}(4,4)(8,4)
\pcline{->}(10,6)(15,6)\taput{$\nmr_{\!2}$}\tbput{\footnotesize $\VVV2$}
\end{picture}
\begin{picture}(16,11)(0,4)
\pcline{->}(0,12)(4,12)
\pcline{->}(4,12)(4,8)
\pcline{->}(4,8)(4,4)
\pcline{->}(4,4)(8,4)
\pcline{->}(10,6)(15,6)\taput{$\nmr_{\!3}$}\tbput{\footnotesize VH}
\end{picture}
\begin{picture}(20,11)(0,4)
\pcline{->}(0,12)(4,12)
\pcline{->}(4,12)(4,8)
\pcline{->}(4,8)(8,8)
\pcline{->}(8,8)(12,8)
\pcline{->}(14,6)(19,6)\taput{$\nmr_{\!2}$}\tbput{\footnotesize VH}
\end{picture}
\begin{picture}(20,11)(0,4)
\pcline{->}(0,12)(4,12)
\pcline{->}(4,12)(8,12)
\pcline{->}(8,12)(8,8)
\pcline{->}(8,8)(12,8)
\pcline{->}(14,6)(19,6)\taput{$\nmr_{\!3}$}\tbput{\footnotesize VH}
\end{picture}
\begin{picture}(0,11)(0,4)
\pcline{->}(0,12)(4,12)
\pcline{->}(4,12)(8,12)
\pcline{->}(8,12)(12,12)
\pcline{->}(12,12)(16,12)
\end{picture}\caption[]{\sf\small Types of the successive steps in the computation of $\nmr_{\!12123212323}(\ww)$ for~$\ww$ of length~$4$: in addition to the six VH-steps, which inexorably approach~$\nm(\ww)$, we find four $\VVV1$-steps and one $\VVV2$-steps; this turns out to be the only possible length-$11$ sequence for length-$4$ words.}
\label{F:Steps}
\end{figure}

\begin{rema*}
In the previous proof, one can observe that the number of VV-steps between two VH-steps is bounded above by~$\FF(\pp-1)$, since columns in the grid have length at most $\pp-1$, and deduce $\FF(\pp) \le \pp(\pp-1)/2 + (\pp(\pp-1)/2+1)\FF(\pp-1)$, which is coarser than~\eqref{E:Function} but sufficient to inductively prove termination.
\end{rema*}

The following result, formulated purely in terms of rewriting systems, is an immediate consequence of Proposition~\ref{P:43Terminating}. 

\begin{coro}\label{P:43Convergent}
Assume that~$(\SS,\RR)$ is a reduced quadratic rewriting system. Define~$\ft : \Pow\SS2 \to \Pow\SS2$ by $\ft(\ww) = \ww'$ for $\ww\fl\ww'$ in~$\RR$ and $\ft(\ww) = \ww$ otherwise.

\ITEM1 Assume that, for all $\rr,\ss,\tt$ in~$\SS$, with $\rr\sep\ss$ not $\RR$-normal, if $\ss\sep\tt$ is $\RR$-normal then $\ft_{12}(\rr\sep\ss\sep\tt)$ is $\RR$-normal, and if $\ss\sep\tt$ is not $\RR$-normal then $\ft_{1212}(\rr\sep\ss\sep\tt)=\ft_{212}(\rr\sep\ss\sep\tt)$. Then~$(\SS,\RR)$ is convergent.

\ITEM2 If $\SS$ contains a $\ft$-neutral element~$\ee$ and the conditions of~\ITEM1 are satisfied for all~$\rr, \ss, \tt$ in~$\SS_\ee$, then~$(\SS,\RR)$ and~$(\SS_{\ee},\RR_{\ee})$ are convergent, where~ $\RR_{\ee}$ consists of one rule $\ww\fl\pi_{\ee}(\ww')$ for each $\ww\fl\ww'$ in~$\RR$ with $\ww\in\SS_{\ee}^*$.
\end{coro}

%%%%
\subsection{Termination in higher classes}\label{SS:Termin44}

We show that Proposition~\ref{P:43Terminating} is optimal: from class $(4, 4)$ onwards, no general termination result can be established, since both nonterminating and terminating rewriting systems may arise.

\begin{prop}\label{P:Termin44}
There exists a quadratic normalisation of class~$(4, 4)$ such that the associated rewriting system is not convergent.
\end{prop}

\begin{proof}
Let $\SS = \{\tta, \ttb, \ttb', \ttb'', \ttc, \ttc', \ttc'', \ttd\}$ and let $\RR$ consist of the five rules $\tta\ttb \to \tta\ttb'$, $\ttb'\ttc' \to \ttb\ttc$, $\ttb\ttc' \to \ttb''\ttc''$, $\ttb'\ttc \to \ttb''\ttc''$, and $\ttc\ttd \to \ttc'\ttd$. We claim that the rewriting system~$(\SS, \RR)$, which is quadratic by definition, is normalising and confluent. However $(\SS, \RR)$ is not terminating, as it admits the length-$3$ cycle 
\[
\underline{\tta\ttb}\ttc\ttd \to \tta\ttb'\underline{\ttc\ttd} \to \tta\underline{\ttb'\ttc'}\ttd \to \tta\ttb\ttc\ttd.
\]
We prove that $(\SS, \RR)$ is normalising using an exhaustive description of the rewriting sequences starting from an arbitrary $\SS$-word. Let~$\ff$ be the accent-forgetting map $\tta \mapsto \tta$; $\ttb, \ttb', \ttb'' \mapsto \ttb$; $\ttc, \ttc', \ttc'' \mapsto \ttc$; $\ttd \mapsto \ttd$. For~$\uu$ a nonempty factor of~$\tta\ttb\ttc\ttd$, we say that an $\SS$-word~$\ww$ is \emph{special of type~$\uu$} if $\ff(\ww) = \uu$ holds. For~$\ww$ in~$\SS^*$, inductively define a decomposition~$\DD(\ww)$ by $\DD(\varepsilon) = \varepsilon$ and, if
$\DD(\ww) = \ww_1 \sep \pdots \sep \ww_\mm$ and~$\ss \in \SS$ hold,
$\DD(\ww\ss) = \ww_1 \sep \pdots \sep \ww_\mm\ss$ if $\ww_\mm\ss$ is special, and $\DD(\ww\ss) = \ww_1 \sep \pdots \sep \ww_\mm \sep \ss$
otherwise: $\DD(\ww)$ is obtained by grouping the special factors of~$\ww$ as much as possible. For instance, we find
$\DD(\tta\ttb'\ttd\ttb''\ttc\ttd\ttb'\tta\ttb') = \tta\ttb'\sep\ttd\sep\ttb''\ttc\ttd\sep\ttb'\sep\tta\ttb'$. Now, we observe that~$\pi$ is compatible with all rules of~$\RR$ and, moreover, every rule acts inside a special factor. Hence, if we have $\DD(\ww) =
\ww_1 \sep \pdots \sep \ww_\mm$, then the words~$\ww'$ for which $\ww \to_\RR \ww'$ holds are those words~$\ww'$ satisfying $\DD(\ww') =
\ww'_1 \sep \pdots \sep \ww'_\mm$ with $\ww_\ii \to_\RR \ww'_\ii$ for
each~$\ii$. Consequently, in order to prove that $(\SS, \RR)$ is normalising and confluent, it suffices to prove it for the factors of the $\DD$-decomposition, that is, for special words. We review the ten types. First, an $\SS$-word~$\ww$ of type~$\tta$, $\ttb$, $\ttc$, $\ttd$, $\tta\ttb$, $\ttb\ttc$, or $\ttc\ttd$ is $\RR$-normal form, or we have $\ww\fl_{\RR}\ww'$ for some $\RR$-normal $\SS$-word~$\ww'$. Next, there are nine $\SS$-words of type~$\tta\ttb\ttc$, and the corresponding restriction of~$\fl_{\RR}$ is (where framed $\SS$-words are the $\RR$-normal ones)
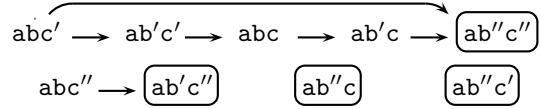
\[
\begin{picture}(70,6)(0,0)
\put(0,0){$\tta\ttb\ttc'$}
\psline{->}(8,0.5)(13,0.5)
\put(15,0){$\tta\ttb'\ttc'$}
\psline{->}(23,0.5)(28,0.5)
\put(30,0){$\tta\ttb\ttc$}
\psline{->}(38,0.5)(43,0.5)
\put(45,0){$\tta\ttb'\ttc$}
\psline{->}(53,0.5)(58,0.5)
\put(60,0){$\tta\ttb''\ttc''$}
\psframe[linewidth=0.8pt,framearc=.5](59,-1.5)(70,4)
\psline(3,3)(3,3)
\psarc{<-}(55,1){4}{40}{90}
\psline(8,5)(55,5)
\psarc(8,1){4}{90}{140}
\end{picture}
\]
\[
\begin{picture}(63,6)(0,-1)
\put(0,0){$\tta\ttb\ttc''$}
\psline{->}(8,0.5)(13,0.5)
\put(15,0){$\tta\ttb'\ttc''$}
\psframe[linewidth=0.8pt,framearc=.5](14,-1.5)(24,4)
\put(35,0){$\tta\ttb''\ttc$}
\psframe[linewidth=0.8pt,framearc=.5](34,-1.5)(43,4)
\put(55,0){$\tta\ttb''\ttc'$}
\psframe[linewidth=0.8pt,framearc=.5](54,-1.5)(64,4)
\end{picture}
\]
The graph for~$\ttb\ttc\ttd$ is entirely similar. Finally, for the type $\tta\ttb\ttc\ttd$, we find:
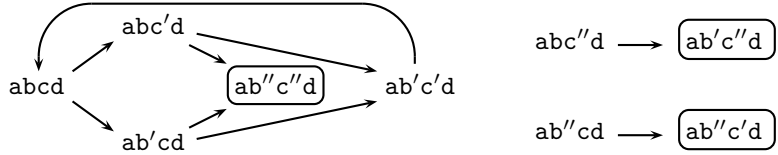
\[
\begin{picture}(110,20)(0,1)
\put(0,8){$\tta\ttb\ttc\ttd$}
\put(15,0.5){$\tta\ttb'\ttc\ttd$}
\put(15,16){$\tta\ttb\ttc'\ttd$}
\put(30,8){$\tta\ttb''\ttc''\ttd$}
\put(50,8){$\tta\ttb'\ttc'\ttd$}
\psframe[linewidth=0.8pt,framearc=.5](29,6.5)(42,11.8)
\psline{->}(8.5,7)(14,3)
\psline{->}(8.5,11)(14,15)
\psline{->}(24,3.5)(29,6)
\psline{->}(24,14.5)(29,12)
\psline{->}(25,2)(49,7)
\psline{->}(25,16)(49,11)
\psline(54,12)(54,13)
\psarc(47,13){7}{0}{90}
\psline(47,20)(11,20)
\psarc(11,13){7}{90}{180}
\psline{->}(4,13)(4,11)
\put(70,2){$\tta\ttb''\ttc\ttd$}
\psline{->}(81,2.5)(87,2.5)
\put(90,2){$\tta\ttb''\ttc'\ttd$}
\psframe[linewidth=0.8pt,framearc=.5](89,0.5)(102,6)
\put(70,14){$\tta\ttb\ttc''\ttd$}
\psline{->}(81,14.5)(87,14.5)
\put(90,14){$\tta\ttb'\ttc''\ttd$}
\psframe[linewidth=0.8pt,framearc=.5](89,12.5)(102,18)
\end{picture}
\]
Thus, for each type, the corresponding connected component of the relation~$\fl_{\RR}$ contains exactly one $\RR$-normal $\SS$-word, which is reachable from any other $\SS$-word of the component. It follows that~$(\SS,\RR)$ is normalising and confluent. Moreover, the inspection of the normalisation of length-three $\SS$-words shows that the normalisation~$(\SS,\nm)$ associated with~$\RR$ is of minimal class~$(4,4)$.
\end{proof}

By contrast, the following example shows that terminating rewriting systems may also arise when the minimal class is at least~$(4,4)$.

\begin{exam}\label{X:Chinois}
For a totally ordered finite set~$\XX$, the \emph{Chinese monoid over~$\XX$} is the monoid~$\CC_{\XX}$ generated by~$\XX$ and submitted to the relations $\ttz\tty\ttx = \ttz\ttx\tty = \tty\ttz\ttx$, for $\ttx\le\tty\le\ttz$~\cite{CassaigneEspieKrobNovelliHivert}. Assume that~$\XX$ has three elements and denote by~$\SS$ the eight-element set obtained from~$\XX$ by adjoining the empty word~$\ee$, the three words $\tty\ttx$ for $\ttx<\tty$, and~$\tty\tty$ if~$\tty$ is the middle element of~$\XX$ (neither the minimal one nor the maximal one). The following twelve rules are derivable from the defining relations of~$\CC_\XX$: the nine rules $\tty\sep\ttx \fl \tty\ttx$, $\tty\sep\tty\ttx \fl \tty\ttx\sep\tty$, $\tty\ttx\sep\ttx \fl \ttx\sep\tty\ttx$ for $\ttx<\tty$; the two rules $\tty\sep\ttz\ttx \fl \ttz\ttx\sep\tty$ and $\ttz\sep\tty\ttx \fl \ttz\ttx\sep\tty$ for $\ttx<\tty<\ttz$; and $\tty\sep\tty \fl \tty\tty$ if~$\tty$ is the middle element of~$\XX$. This reduced rewriting system terminates (using the weighted right-lexicographic order generated by $\ttx<\tty\ttx$ for $\ttx\le\tty$ and $\ttz\ttx<\tty$ for $\ttx<\tty$) and, after application of Knuth-Bendix completion, it yields a convergent rewriting system~$(\SS_{\ee},\RR_{\ee})$ with~$22$ rules presenting~$\CC_\XX$. After homogenisation, we obtain a reduced, quadratic and convergent rewriting system~$(\SS,\RR)$, whose corresponding quadratic normalisation is of class~$(4,4)$, the worst case being reached on $\ttz\sep\tty\tty\sep\tty$ if~$\tty$ is the middle element and~$\ttz>\tty$ holds. Similar convergent quadratic presentations also exist when~$\XX$ has four or five elements (to be compared with the nonquadratic ones of~\cite[Th.~3.3]{CainGrayMalheiro2}), the class being~$(5,4)$ in both cases.
\end{exam}

%%%%
\section{Garside normalisation}\label{S:Garside}

In this last section, we investigate the connection between our current general framework and Garside families. It turns out that the latter provide natural examples of quadratic normalisations of class~$(4, 3)$ and that, conversely, a normalisation of class~$(4, 3)$ comes from a Garside family if, and only if, it satisfies some explicit additional condition called left-weightedness.

The section is organised as follows. In Subsection~\ref{SS:Greedy}, we briefly recall the basic definitions involving Garside families and the associated normal forms. In Subsection~\ref{SS:LeftWeighted}, we introduce the notion of a left-weighted normalisation and establish the above mentioned connection, which is Theorem~C of the introduction. Finally, in Subsection~\ref{SS:Applic}, we mention a few further consequences. 

%%%%
\subsection{Greedy decompositions}
\label{SS:Greedy}

Hereafter, if $\MM$ is a left-cancellative monoid, we denote by~$\dive$ the associated \emph{left-divisibility} relation, defined by $\ff\dive\gg$ if $\ff\gg' = \gg$ holds in~$\MM$ for some~$\gg'$. The starting point is the notion of an $\SS$-normal word.

\begin{defi}[{\cite[Def.~III.1.1]{Garside}}]
\label{D:Greedy}
If~$\MM$ is a left-cancellative monoid and~$\SS$ is included in~$\MM$, an $\SS$-word~$\ss_1\sep\ss_2$ is called \emph{$\SS$-normal} if the following condition holds:
\begin{equation}\label{E:Greedy}
\forall\ss{\in}\SS\: \forall\ff{\in}\MM\: (\ss \dive \ff \ss_1\ss_2 \Rightarrow \ss\dive \ff \ss_1).
\end{equation}
An $\SS$-word $\ss_1 \sep \pdots \sep \ss_\pp$ is called \emph{$\SS$-normal} if $\ss_\ii \sep \ss_{\ii+1}$ is $\SS$-normal for every~$\ii$.
\end{defi}

The intuition underlying condition~\eqref{E:Greedy} is that~$\ss_1$ already contains as much of~$\SS$ as it can, a greediness condition; note that we do not only consider the left-divisors of~$\ss_1 \ss_2$ that lie in~$\SS$, but, more generally, all elements of~$\SS$ that left-divide~$\ff \ss_1 \ss_2$.

Then the notion of a Garside family arises naturally. Here we state the definition in a restricted case fitting our current framework (see~\cite{Garside} for the general case):

\begin{defi}\label{D:GarNormal}
Assume that $\MM$ is a monoid with no nontrivial invertible elements and $\SS$ is a subset of~$\MM$ that contains~$1$. We say that $\SS$ is a \emph{Garside family in~$\MM$} if every element~$\gg$ of~$\MM$ has an $\SS$-normal decomposition, that is, there exists an $\SS$-normal $\SS$-word $\ss_1 \sep \pdots \sep \ss_\pp$ satisfying $\ss_1 \pdots \ss_\pp = \gg$.
\end{defi}

\begin{exam}\label{X:Artin}
The seminal example of a Garside family is the family of all simple braids. Let~$B_\nn$ be Artin's $\nn$-strand braid group and~$\BP\nn$ be the submonoid of~$B_\nn$ consisting of all braids that can be represented by a diagram in which all crossings have a positive orientation (see for instance~\cite{Gar} or \cite[Section~I.1]{Garside}). Then the subfamily~$\SS_\nn$ of~$\BP\nn$ consisting of those positive braids that can be represented by a diagram in which any two strands cross at most once is a Garside family in~$\BP{\nn}$.

More generally, if~$\MM$ is an Artin--Tits monoid, that is, a monoid defined by relations of the form 
$\ss\tt\ss\tt... = \tt\ss\tt\ss...$
where both terms have the same length, and if~$\WW$ is the Coxeter group obtained by adding the torsion relations~$\ss^2 = 1$ to the above relations, then~$\MM$ admits a Garside family that is a copy of~$\WW$~\cite{Din}. When~$\WW$ is finite, this Garside family (which consists of the divisors of some element~$\Delta$ connected with the longest element of~$\WW$) is minimal. When~$\WW$ is infinite, it is not minimal, but there exists in every case a finite Garside family~\cite{Din}. For instance, if $\MM$ is the Artin--Tits monoid of (affine) type~$\Att$, that is, $\MM$ admits a presentation with three generators $\sig1, \sig2, \sig3$ and three relations $\sig\ii\sig\jj\sig\ii = \sig\jj\sig\ii\sig\jj$, then the associated Coxeter group is infinite, but $\MM$ admits a finite Garside family~$\SS$ consisting of the sixteen right-divisors of the elements $\sig1\sig2\sig3\sig2$, $\sig2\sig3\sig1\sig3$, and $\sig3\sig1\sig3\sig1$.
\end{exam}

It turns out that a large number of monoids admit interesting Garside families, and many results involving such families, including various practical characterisations, and the derived normalisations are now known~\cite{Garside}.

For our current approach, what counts is that Garside normalisation enters the framework of Sections~\ref{S:Normal} to~\ref{S:Class3}. First, a mild discussion is in order, because the $\SS$-normal form as introduced in Definition~\ref{D:GarNormal} is not readily unique. 

\begin{lemm}\label{L:GarNormal}
Assume that~$\MM$ is a left-cancellative monoid with no nontrivial invertible elements and~$\SS$ is a Garside family in~$\MM$. 

\ITEM1 \cite[Prop.~III.1.25]{Garside} Call two $\SS$-words $\simeq$-equivalent if they only differ by appending final entries~$1$. Then every $\SS$-word that is $\simeq$-equivalent to an $\SS$-normal word is $\SS$-normal; conversely, any two $\SS$-normal decompositions of the same element of~$\MM$ are $\simeq$-equivalent.

\ITEM2 \cite[Prop.~III.1.30]{Garside} Every element of~$\MM$ with a representative in~$\Pow{\SS}{\pp}$ admits an $\SS$-normal decomposition of length at most~$\pp$.
\end{lemm}

Building on Lemma~\ref{L:GarNormal}, we immediately obtain

\begin{prop}\label{P:GarNormal}
Assume that~$\MM$ is a left-cancellative monoid with no nontrivial invertible elements and~$\SS$ is a Garside family of~$\MM$. Then every element~$\gg$ of~$\MM$ admits a unique $\SS$-normal decomposition of minimal length, and the corresponding map is a geodesic normal form on~$(\MM, \SS\setminus\{1\})$.
\end{prop}

We can then apply Proposition~\ref{P:GenNF}, and associate with the Garside family~$\SS$ a normalisation~$(\SS, \nm)$. The latter involves the generating set~$\SS\setminus\{1\}$ enriched with one letter representing the unit, and it is then natural to use~$1$ for that letter so that we simply recover~$\SS$. We shall then say that $(\SS, \nm)$ is \emph{derived from the Garside family~$\SS$}. In this case, $1$ is an $\nmS$-neutral element by Proposition~\ref{P:GenNF}, and $\MM$ admits the presentation~\eqref{E:Neutral2}. Here is the main observation:

\begin{prop}\label{P:GarClass}
Assume that $\MM$ is a left-cancellative monoid with no nontrivial invertible elements and~$\SS$ is a Garside family of~$\MM$. Then the normalisation~$(\SS, \nmS)$ derived from~$\SS$ is quadratic of class~$(4, 3)$.
\end{prop}

\begin{proof}
That $\nmS$ satisfies~\eqref{E:Quad1} directly follows from Definition~\ref{D:Greedy}, since $\SS$-greedy words are defined by a condition that only involves length-two factors. For~\eqref{E:Quad2} and the more precise result about the class, it follows from~\cite[Prop.~III.1.45]{Garside} which states that the domino rule is valid for $\nmrS$. As the current statement is different from \cite[Prop.~III.1.45]{Garside}, we recall the argument.

\rightskip40mm
So assume that $\ss_1, \ss_2, \ss'_1, \ss'_2, \tt_0, \tt_1, \tt_2$ lie in~$\SS$, that $\ss_1 \sep \ss_2$ is $\SS$-normal, and that we have $\ss'_1 \sep \tt_1 = \nmS(\tt_0 \sep \ss_1)$ and $\ss'_2 \sep \tt_2 = \nmS(\tt_1 \sep \ss_2)$. Assume $\ss \in \SS$ and $\ss \dive \ff \ss'_1 \ss'_2$. A fortiori we have $\ss \dive \ff \ss'_1 \ss'_2 \tt_2$, hence $\ss \dive \ff \tt_0 \ss_1 \ss_2$, since the diagram on the right is commutative. As $\ss_1 \sep \ss_2$ is $\SS$-normal, we deduce $\ss \dive \ff\tt_0 \ss_1$, whence $\ss \dive \ff \ss'_1 \tt_1$. As $\ss'_1 \sep \tt_1$ is $\SS$-normal, we deduce $\ss \dive \ff \ss'_1$. This shows that $\ss'_1 \sep \ss'_2$ is $\SS$-normal.\hfill$\square$%
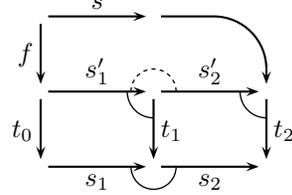
\begin{picture}(0,0)(-6,-2)
\psarc[style=thin](15,0){3}{180}{360}
\psarc[style=thin](15,10){3.5}{180}{270}
\psarc[style=thin](30,10){3.5}{180}{270}
\psarc[style=thinexist](15,10){3}{0}{180}
\pcline{->}(1,0)(14,0)\tbput{$\ss_1$}
\pcline{->}(16,0)(29,0)\tbput{$\ss_2$}
\pcline{->}(1,10)(14,10)\taput{$\ss'_1$}
\pcline{->}(16,10)(29,10)\taput{$\ss'_2$}
\pcline{->}(1,20)(14,20)\taput{$\ss$}
\pcline{->}(0,19)(0,11)\tlput{$\ff$}
\pcline{->}(0,9)(0,1)\tlput{$\tt_0$}
\pcline{->}(15,9)(15,1)\trput{$\tt_1$}
\pcline{->}(30,9)(30,1)\trput{$\tt_2$}
\pcline(16,20)(23,20)\psarc(23,13){7}{0}{90}
\pcline{->}(30,13)(30,11)
\end{picture}
\def\qed{\relax}\end{proof}

The result of Proposition~\ref{P:GarClass} is optimal: as the example below shows, $(4, 3)$ is in general the minimal class. Let us mention that there is a particular class of Garside families, called bounded~\cite[chapter~VI]{Garside}, for which the class drops to~$(3, 3)$ or less. Garside monoids~\cite{Dgk} are typical examples of the latter situation. 

\begin{exam}
The normalisation derived from the finite Garside family mentioned in Example~\ref{X:Artin} for the Artin--Tits monoid of type~$\Att$ is not of class~$(3, 3)$: for instance, one finds $\nmr_{\!121}(\sig1\sep\sig1\sig2\sep \sig1\sig3) = \sig1\sig2\sig1 \sep \sig2 \sep \sig3$, in which $\sig2 \sep \sig3$ is not normal.
\end{exam} 

%%%%
\subsection{Left-weighted normalisation}\label{SS:LeftWeighted}

Definition~\ref{D:Greedy} is highly non-symmetric, so we can expect that the normalisations derived from Garside families satisfy some relations capturing the specific role of the left-hand side. 

\begin{defi}
Assume that $(\SS, \nm)$ is a (quadratic) normalisation for a monoid~$\MM$. We say that $(\SS, \nm)$ is \emph{left-weighted} if, for all $\ss, \tt, \ss', \tt'$ in~$\SS$, the equality $\ss'\vert\tt' = \nm(\ss \sep \tt)$ implies $\ss \dive \ss'$ in~$\MM$. 
\end{defi}

In other words, a normalisation~$(\SS, \nm)$ is left-weighted if, for every~$\ss$ in~$\SS$, the first entry of any length-two $\SS$-word~$\nm(\ss \sep \tt)$ is always a right-multiple of~$\ss$ in the associated monoid: normalising~$\ss\vert\tt$ amounts to adding something in the \emph{left} entry.

\begin{lemm}\label{L:LeftWeighted}
The normalisation derived from a Garside family in a left-cancellative monoid with no nontrivial invertible element is left-weighted.
\end{lemm}

\begin{proof}
Assume that $(\SS, \nmS)$ derives from a Garside family~$\SS$. If $\ss'\vert\tt' = \nmS(\ss\vert\tt)$ holds, $\ss$ is an element of~$\SS$, and we have $\ss \tt = \ss' \tt'$, whence $\ss \dive \ss' \tt'$. By assumption, $\ss' \sep \tt'$ is $\SS$-normal, so~\eqref{E:Greedy} implies $\ss \dive \ss'$. Thus~$\nmS$ is left-weighted.
\end{proof}

We shall now establish that left-weightedness characterises Garside normalisation, as stated in Theorem~C:

\begin{prop}\label{P:GarCharac}
Assume that~$(\SS, \nm)$ is a quadratic normalisation mod~$\one$ for a monoid~$\MM$ that is left-cancellative and contains no nontrivial invertible element. Then the following are equivalent:

\ITEM1 The family~$\SS$ is a Garside family in~$\MM$ and~$(\SS, \nm)$ derives from it.

\ITEM2 The normalisation~$(\SS, \nm)$ is of class~$(4, 3)$ and is left-weighted.
\end{prop}

According to Proposition~\ref{P:GarClass} and Lemma~\ref{L:LeftWeighted}, the implication \ITEM1$\Rightarrow$\ITEM2 holds and we are left with the converse direction. So, until the end of the subsection, we assume that $(\SS, \nm)$ is a left-weighted quadratic normalisation mod~$\one$ of class~$(4, 3)$ for a monoid~$\MM$ that is left-cancellative and contains no nontrivial invertible element. Our aim is to show that~$\SS$ is a Garside family in~$\MM$ and that~$\nm$ derives from it. We decompose the argument into several steps.

\begin{lemm}\label{L:RDClosed}
The family~$\SS$ is closed under right-divisor in~$\MM$.
\end{lemm}

\begin{proof}
Assume $\ss \in \SS$. An element~$\gg$ of~$\MM$ is a right-divisor of~$\ss$ if there exists~$\ff$ in~$\MM$ satisfying $\ss = \ff\gg$. We prove~$\gg\in\SS$ by induction on the minimal length~$\LG\ff$ of the $\SS$-words representing~$\ff$. For $\LG\ff = 0$, we have $\ss = \gg$, whence $\gg \in \SS$. 

Assume $\LG\ff\ge 1$. Then we can write $\ff = \ff' \tt$ for some~$\ff'$ satisfying $\LG{\ff'} = \LG\ff-1$ and some~$\tt$ in~$\SS$. Then we have $\ss = \ff'\tt\gg$, so the induction hypothesis implies $\tt\gg \in \SS$ and, therefore, the $\SS$-normal decompositions of~$\tt\gg$ are the $\SS$-words $\tt\gg \sep 1 \sep \pdots \sep 1$. Let $\ss_1 \sep \pdots \sep \ss_\pp$ be an $\nm$-normal decomposition of~$\gg$. For $\pp = 1$, we have $\gg = \ss_1 \in \SS$. So assume $\pp \ge 2$. By Proposition~\ref{P:NormClass3}, as~$\nm$ is quadratic of class~$(4, 3)$, the domino rule is valid for~$\nmr$ and, therefore, an $\nm$-normal decomposition of~$\tt \gg$ is $\ss'_1 \sep \pdots \sep \ss'_\pp \sep \tt_\pp$ where we put $\tt_0 = \tt$ and $\ss'_\ii \sep \tt_\ii = \nm(\tt_{\ii-1} \sep \ss_\ii)$ for $\ii = 1 \wdots \pp$.
By $\simeq$-uniqueness of the $\nm$-normal form, we have $\ss'_1 = \tt\gg$ and $\ss'_2 = \pdots = \ss'_\pp = \tt_\pp = 1$. 
Now, we prove using induction on~$\kk$ decreasing from~$\pp$ to~$2$ that $\tt_{\kk-1}$ and~$\ss_\kk$ equal~$1$. For $\kk = \pp$, we have $\ss'_{\pp}=\tt_{\pp}=1$; by construction, we have $\tt_{\pp-1} \ss_\pp = \ss'_\pp \tt_\pp$, whence $\tt_{\pp-1} \ss_\pp = 1$, and $ \tt_{\pp-1} = \ss_\pp = 1$, since~$\MM$ contains no nontrivial invertible element. For $2 \le \kk < \pp$, we have $\ss'_\kk = 1$ by assumption and $\tt_\kk = 1$ by induction hypothesis, so that the same argument gives $ \tt_{\kk-1} = \ss_\kk = 1$. Thus~$\gg$ admits an $\nm$-normal decomposition of the form $\ss_1 \sep 1 \sep \pdots \sep 1$ and, therefore, it belongs to~$\SS$.
\end{proof}

\begin{lemm}\label{L:Head}
For~$\gg$ in~$\MM$, define~$\HH(\gg)$ to be~$1$ for~$\gg = 1$, and to be the first entry in the $\nm$-normal decomposition of~$\gg$ otherwise. Then $\HH(\gg)$ is an element of~$\SS$ that left-divides~$\gg$, and every element of~$\SS$ that left-divides~$\gg$ in~$\MM$ left-divides~$\HH(\gg)$.
\end{lemm}

\begin{proof}
By definition, $\HH(\gg)$ belongs to~$\SS$, and it left-divides~$\gg$ in~$\MM$, since we have $\gg = \HH(\gg) \ev(\ww)$ if $\HH(\gg) \sep \ww$ is the $\nm$-normal decomposition of~$\gg$. 

Now assume that $\tt$ is an element of~$\SS$ that left-divides~$\gg$, say $\gg = \tt \hh$. Let $\ss_1 \sep \pdots \sep \ss_\pp$ be the $\nm$-normal decomposition of~$\hh$. As $\nm$ is quadratic of class~$(4, 3)$, the domino rule is valid for~$\nm$, so the $\nm$-normal decomposition of~$\gg$ is $\ss'_1 \sep \pdots \sep \ss'_\pp \sep \tt_\pp$ with $\tt_0 = \tt$ and $\ss'_\ii \sep \tt_\ii = \nm(\tt_{\ii-1} \sep \ss_\ii)$ for $\ii = 1 \wdots \pp$. By uniqueness of the $\nm$-normal form, we must have $\HH(\gg) = \ss'_1$. But the fact that $(\SS, \nm)$ is left-weighted implies that $\tt_0$ left-divides the first entry in~$\nm(\tt_0 \sep \ss_1)$, which is $\ss'_1\sep\tt_1$, so $\tt\dive\HH(\gg)$ holds.
\end{proof}

\begin{lemm}
The family~$\SS$ is a Garside family in the monoid~$\MM$.
\end{lemm}

\begin{proof}
By assumption, $\SS$ contains~$1$ and, by Lemma~\ref{L:RDClosed}, it is closed under right-divisor. So~$\SS$ is what is called solid in~\cite[Section~IV.2]{Garside}. Moreover, by definition, $\SS$ is a generating family in~$\MM$. Then, by~\cite[Prop.~IV.2.7]{Garside}, we know that a solid generating family is a Garside family in~$\MM$ if, and only if, for every element~$\gg$ of~$\MM$, there exists an element~$\HH(\gg)$ of~$\SS$ with the properties of Lemma~\ref{L:Head}. Thus the latter lemma implies that~$\SS$ is a Garside family in~$\MM$. 
\end{proof}

We can now complete the argument.

\begin{proof}[Proof of Proposition~\ref{P:GarCharac}]
Owing to the previous results, it only remains to show that, in the implication \ITEM2$\Rightarrow$\ITEM1, the given normalisation~$(\SS, \nm)$ coincides with the one, say~$(\SS, \nm')$, derived from the Garside family~$\SS$. As both normalisations are quadratic, it is sufficient to prove $\nm(\ss_1 \sep \ss_2)=\nm'(\ss_1\sep\ss_2)$ for all $\ss_1,\ss_2\in\SS$, and, to this end, it is sufficient to prove that $\nm(\ss_1 \sep \ss_2)$ is $\SS$-normal (in the sense of Definition~\ref{D:Greedy}). Now assume $\ss'_1\sep\ss'_2=\nm(\ss_1 \sep \ss_2)$. Since~$\SS$ is a Garside family in~$\MM$, we can appeal to~\cite[Corollary~IV.1.31]{Garside} which says that~$\ss'_1 \sep \ss'_2$ is $\SS$-greedy if, and only if, every element of~$\SS$ that left-divides~$\ss'_1 \ss'_2$ left-divides~$\ss'_1$: we can skip the term~$\ff$ in~\eqref{E:Greedy}. Now assume $\ss \in \SS$ and $\ss \dive \ss'_1 \ss'_2 = \ss_1 \ss_2$. By Lemma~\ref{L:Head}, we have $\tt \dive \HH(\ss_1 \ss_2) = \ss'_1$ and, therefore, $\ss'_1 \sep \ss'_2$ is $\SS$-greedy. 
\end{proof}

\begin{rema*}
In Proposition~\ref{P:GarCharac}, we take as an assumption that the monoid~$\MM$ associated with~$(\SS, \nm, \ee)$ is left-cancellative and has no nontrivial invertible element. It is natural to wonder whether explicit conditions involving~$(\SS, \nm, \ee)$ imply these assumptions. For invertible elements, requiring that $\ss\sep\tt \not= \ee\sep\ee$ implies $\nm(\ss\sep\tt) \not= \ee\sep\ee$ is such a condition but, for left-cancellativity, we leave it as an open question.
\end{rema*}

%%%%
\subsection{Two further results}\label{SS:Applic}

By Proposition~\ref{P:QuadPres}, if $(\SS, \nm)$ is a quadratic normalisation for a monoid~$\MM$, then $\MM$ admits a presentation consisting of all quadratic relations $\ss\sep\tt =\ss'\sep\tt'$, with $\ss'\sep\tt'=\nm (\ss \sep \tt)$. In fact, in the left-weighted case, this presentation can be replaced with a smaller one involving triangular relations of the form $\rr\sep\ss = \tt$.

\begin{prop}\label{P:Triangle}
Assume that $(\SS, \nm)$ is a left-weighted quadratic normalisation system of class~$(4, 3)$ mod~$\ee$ for a left-cancellative monoid~$\MM$. Then~$\MM$ admits the presentation $(\SSm, \TT)$ where~$\TT$ consists of all relations~$\ss\sep\tt = \ss\tt$ with $\ss,\tt$ in~$\SSm$ satisfying $\ss\tt\in\SS$.
\end{prop}

\begin{proof}
By Proposition~\ref{P:QuadPres}~\ITEM1, we know that $\MM$ admits a presentation in terms of~$\SS$ by the relations $\ss\sep\tt = \pi_{\ee}(\nm(\ss\sep\tt))$ with $\ss, \tt\in\SSm$. First, if $\ss,\tt$ in~$\SSm$ satisfy $\ss\tt\in\SSm$, then we must have $\nm(\ss\sep\tt)=\ss\tt\sep\ee$, and, if they satisty $\ss\tt=1$, then we must have $\nm(\ss\sep\tt)=\ee\sep\ee$, so that $\pi_{\ee}(\nm(\ss\sep\tt))=\ss\tt$ holds in both cases. Thus, $\TT$ is included in the presentation of Proposition~\ref{P:QuadPres}~\ITEM1. Conversely, let us show that each relation $\ss\sep\tt = \pi_{\ee}(\nm(\ss\sep\tt))$ with $\ss,\tt$ in~$\SSm$ follows from a finite number of relations of~$\TT$. So assume that~$\ss$ and~$\tt$ lie in~$\SSm$ and let $\ss' \sep \tt' = \nm(\ss \sep \tt)$. If $\tt'=\ee$ holds, we have $\ss'=\ss\tt$ in~$\MM$, so the result is true. Otherwise, the assumption that $(\SS, \nm)$ is left-weighted implies that there exists~$\rr$ in~$\MM$ satisfying $\ss \rr = \ss'$. By construction, $\rr$ is a right-divisor of~$\ss'$ in~$\MM$ so, by Lemma~\ref{L:RDClosed}, $\rr$ must lie in~$\SS$. Then, in~$\MM$, we have $\ss' = \ss \rr$, whence $\ss \tt = \ss \rr \tt'$. The assumption that~$\MM$ is left-cancellative implies $\tt = \rr \tt'$. Hence the relation $\ss\sep\tt = \ss'\sep\tt'$ is the consequence of $\ss\sep\rr = \ss'$ and $\rr \tt' = \tt$.
\end{proof}

Note that the existence of the presentation of Proposition~\ref{P:Triangle} is only possible in a non-graded context, except for the free monoid~$\SS^*$ with its presentation $\PRESp{\SS}{}$.

\begin{exam}
Consider the braid monoid~$\BP3$, that is, the monoid presented by $\PRESp{\tta, \ttb}{\tta\ttb\tta = \ttb\tta\ttb}$. Then $\BP3$ has a Garside family consisting of the six elements $1$, $\tta$, $\ttb$, $\tta\ttb$, $\ttb\tta$, and~$\tta\ttb\tta$. Proposition~\ref{P:Triangle} provides a presentation of~$\BP3$ whose generators are the five nontrivial elements of the Garside family, and with the six relations $\tta \sep \ttb = \tta\ttb$, $\ttb \sep \tta = \ttb\tta$, and $\tta\sep\ttb\tta=\ttb\sep\tta\ttb=\tta\ttb\sep\tta=\ttb\tta\sep\ttb=\tta\ttb\tta$. This presentation is much smaller than the one provided by~\eqref{E:QuadNeutral2}, that has the same five generators and $5^2=25$ relations, such as $\tta\ttb\sep\tta\ttb=\tta\ttb\tta\sep\ttb$ or $\tta\sep\tta=\tta\sep\tta$. 
\end{exam}

Another subsequent development is that Garside families give rise to convergent rewriting system. Indeed, Propositions~\ref{P:43Terminating} and~\ref{P:GarCharac} directly imply

\begin{prop}\label{P:GarTermination}
Assume that~$\MM$ is a left-cancellative monoid with no nontrivial invertible elements and~$\SS$ is a Garside family in~$\MM$. Let~$\RR$ consist of all rules $\ss\sep\tt\fl\ww$ with $\ss,\tt$ in~$\SS\setminus\{1\}$ and~$\ww$ the minimal length $\SS$-normal decomposition of~$\ss\sep\tt$. Then the rewriting system $(\SS\setminus\{1\},\RR)$ is convergent.
\end{prop}

As mentioned in Example~\ref{X:Artin}, every finitely generated Artin--Tits monoid admits a finite Garside family. Being also left-cancellative with no nontrivial invertible element, Artin--Tits monoids are thus eligible to Proposition~\ref{P:GarTermination}.

\begin{coro}\label{C:ArtinConvergent}
Every Artin--Tits monoid admits a finite quadratic convergent presentation.
\end{coro}

\begin{exam}
In the case of a spherical Artin--Tits monoid, the elements of the corresponding Coxeter group form a finite Garside family, and Corollary~\ref{C:ArtinConvergent} corresponds to~\cite[Th.~3.1.3, Prop.~3.2.1]{GaussentGuiraudMalbos}. In the nonspherical case, Corollary~\ref{C:ArtinConvergent} is an improvement of the latter results, which only give an infinite convergent presentation. For instance, in type~$\Att$, the $16$-element Garside family described in Example~\ref{X:Artin} yields a convergent rewriting system for~$\Att$ with~$15$ generators and~$87$ relations.
\end{exam}

%%%%

\end{document}